\documentclass[reqno,12pt]{amsart}
\usepackage{amssymb, amsmath,latexsym,amsfonts,amsbsy, amsthm}
\usepackage{color,esint,pgfplots}
\usepackage{cases}
\usepackage{mathtools}
\usepackage{hyperref}
\usepackage{enumitem}
\usepackage{setspace}
\usepackage{mathrsfs}
\usepackage{stmaryrd}
\usepackage{graphicx}
\allowdisplaybreaks[4]
\usepackage[left=2cm, right=2cm, top=3cm, bottom=2cm]{geometry}
\newcommand{\picdis}[1]{}
\let\e=\varepsilon
 \newcommand{\ve}{\varepsilon}

\newcommand{\p}{\partial}

\newcommand{\dd}{\mathrm{d}}

\newcommand{\Q}{\mathcal{Q}}
\newcommand{\N}{\mathcal{N}}
\newcommand{\R}{\mathbb{R}}

\newcommand{\uu}{\mathbf{u}}
\newcommand{\bxi}{\boldsymbol{\xi}}

 \newcommand{\BS}{{\mathbb{S}^2}}

\newcommand{\HH}{\mathbf{H}}

\newcommand{\nn}{\mathbf{n}}
 
\newcommand{\vv}{\mathbf{v}}

\renewcommand{\O}{\Omega}
\renewcommand{\[}{\begin{equation}}
\renewcommand{\]}{\end{equation}}
\renewcommand{\(}{\left(}
\renewcommand{\)}{\right)}

\newcommand{\tr}{\operatorname{tr}}
\renewcommand{\div}{\operatorname{div}}

\newtheorem{theorem}{Theorem}[section]
\newtheorem{definition}[theorem]{Definition}
\newtheorem{lemma}[theorem]{Lemma}
\newtheorem{proposition}[theorem]{Proposition}

\newtheorem{coro}[theorem]{Corollary}

\begin{document}
\title
{Nematic-Isotropic phase transition in Beris-Edward system at critical temperature}

\author{Xiangxiang Su}
\address{School of Mathematical Sciences, Shanghai Jiao Tong University, Shanghai 200240, P. R. China.}
\email{sjtusxx@sjtu.edu.cn}
\begin{abstract}We are concerned with the sharp interface limit for the Beris-Edward system in a bounded domain $\Omega \subset \mathbb{R}^3$ in this paper. The system can be described as the incompressible Navier-Stokes equations coupled with an evolution equation for the Q-tensor. We prove that the solutions to the Beris-Edward system converge to the corresponding solutions of a sharp interface model under well-prepared initial data, as the thickness of the diffuse interfacial zone tends to zero. Moreover, we give not only the spatial decay estimates of the velocity vector field in the $H^1$ sense but also the error estimates of the phase field. The analysis relies on the relative entropy method and elaborated energy estimates.

~
\end{abstract}

\maketitle

\numberwithin{equation}{section}

\setcounter{secnumdepth}{3}
\tableofcontents

\indent


\section{Introduction and Main Results}

Nematic liquid crystals are a special phase of liquid crystals, whose molecular alignment exhibits a slight degree of orderliness. There are various theoretical models of nematic liquid crystals, and a lot of literature explores the relationship among these theories, such as \cite{maiersaupe, Frank, Leslie, Oseen, xin}. 
In this paper, we consider a nematic liquid crystal described by the Beris-Edward  system.  More specifically, we are concerned with the sharp interface limit of the following system in a smooth bounded domain $\Omega \subset \mathbb{R}^3$:
\begin{subequations}  \label{9.14.24}
\begin{align}
\partial_t\mathbf{v}_\varepsilon+(\mathbf{v}_\varepsilon \cdot\nabla) \mathbf{v}_\varepsilon-\Delta \mathbf{v}_{\varepsilon}+\nabla p_{\varepsilon} &=-\varepsilon \operatorname{div} (\nabla Q_{\varepsilon} \odot \nabla Q_{\varepsilon}) && \text { in } \Omega \times(0, T_{1}), \label{9.6.4}\\
\operatorname{div} \mathbf{v}_{\varepsilon} &=0 && \text { in } \Omega \times (0, T_{1}), \\
\partial_{t} Q_{\varepsilon}+(\mathbf{v}_{\varepsilon} \cdot \nabla) Q_{\varepsilon} &=\Delta Q_{\varepsilon}-\frac{1}{\varepsilon^{2}} D F (Q_{\varepsilon}) && \text { in } \Omega \times (0, T_{1}), \label{10.19.1}\\
\mathbf{v}_{\varepsilon}|_{\partial \Omega} = 0, \qquad &  Q_{\varepsilon}|_{\partial \Omega}=0 &&\text { on } \partial \Omega \times (0, T_{1}),\label{bc of omega} \\
\mathbf{v}_{\varepsilon}|_{t=0} =\mathbf{v}_{0, \varepsilon}, \qquad & Q_{\varepsilon}|_{t=0} =Q_{0, \varepsilon} && \text { in } \Omega, \label{9.6.5}
\end{align}
\end{subequations}
where $\mathbf{v}_{\varepsilon}$ and $ p_{\varepsilon}$ denote the velocity vector and pressure of the fluid respectively. And $\varepsilon$ is a small positive parameter which represents the relative strength of elastic and bulk energy. 

$Q_{\varepsilon}$ denotes the order parameter and defined as a symmetric and traceless $3\times 3$ matrix \cite{Ball J M}.
In the Landau-De Gennes theoretical framework \cite{DeGennesProst1995}, order-parameter Q-tensor is defined as follows:
\begin{equation}
Q(x)=\int_{\mathbb{S}^2} (p\otimes p-\tfrac13I_3)f(x,p)\, dp,
\end{equation}
and it quantifies the deviation of the second moment tensor from its isotropic value. Moreover, $f(x,p)$  provides the probability that the molecules, whose center of mass is in a small neighbourhood of the point $x$, are oriented in the direction $p \in \BS$(cf. \cite{maiersaupe}).  
 The configuration space of it is represented as the $5$-dimensional linear space 
\begin{equation}\label{Q-space}
  \Q=\{Q\in \R^{3\times 3}\mid Q =Q^T,~\tr Q =0\}.
\end{equation}

Nematic liquid crystals can be further divided into uniaxial and biaxial nematic liquid crystals. 
When the Q tensor has two equal non-zero eigenvalues $-\frac{s}{3}$, it is called uniaxial. In this case, the Q tensor can be written in a special diagonal form, forming a  $3$-dimensional   manifold  in $\Q$, usually as follows:
\begin{equation}\label{uniaxial}
\mathcal{U}\triangleq \left\{Q\in \Q~\Big|~ Q=s\left(\mathbf{e} \otimes \mathbf{e}-\frac{1}{3} I_3\right)\quad \text{for some}~ s \in \mathbb{R}~\text{and}~ \mathbf{e}  \in \BS\right\}.
\end{equation}
For the introduction to the biaxial case, the readers are encouraged to refer to other references, such as \cite{Majumdar}. 

In \eqref{9.6.4}, the term $\nabla Q_{\varepsilon} \odot \nabla Q_{\varepsilon}$ represents the $3 \times 3$ matrix, where each element in the $(i, j)$-th position corresponds to the dot product of the gradients $\partial_{x_i} Q_{\varepsilon} : \partial_{x_j} Q_{\varepsilon}$, where $1 \le i,j \le 3$. 
And an important fact about $Q_{\varepsilon}$ is that
\begin{align} \label{basic constant1}
\|Q_\ve\|_{L^\infty(\Omega\times(0,T))}\leq c_0(a,b,c,\|Q_{0, \varepsilon}\|_{\mathrm{L}^{\infty}(\Omega)}) ,
\end{align}
which will be proved in Lemma \ref{L infinity bound}.

$D  F(Q)$ in \eqref{10.19.1} means the variation of $F(Q)$ in space $\Q$, where $F(Q)$ is the bulk energy density used to describe the bulk effect. It is  usually expressed as a fourth-order polynomial with respect to Q. A typical example takes the form:
\begin{equation}\label{bulkED}
F(Q)=
\frac{a}{2}\mathrm{tr}(Q^2)-\frac{b}{3}\tr(Q^3)+\frac{c}{4}\(\tr(Q^2)\)^2,
\end{equation} 
where $a, b$ and $c$ are positive constants that depend on the material properties and temperature. 
Then $D  F(Q)$ can be expressed as
\begin{equation}\label{pF}
(D F(Q))_{ij}=a Q_{i j}-b \sum_{k=1}^3Q_{i k} Q_{k j}+c|Q|^{2} Q_{i j}+\frac{b}{3} |Q|^{2} \delta_{i j}.
\end{equation}

$F(Q)$ is also related to a free energy associated with the orientation of liquid crystal molecules, which is denoted by
\begin{equation}\label{GL energy}
 \mathcal{E}_\ve(Q)=\int_{\O} \(\frac\ve2 |\nabla Q|^2+\frac{1}\ve F(Q)  \)\, dx.
\end{equation}
This integration occurs in a smooth bounded domain $\Omega \subset \mathbb{R}^3$. Moreover, $|\nabla Q|=\sqrt{\sum_{ijk}|\p_k Q_{ij}|^2}$. 
The stationary points of $F(Q)$ correspond to uniaxial (cf. \cite{Majumdar0}). In this case,
 \begin{equation}\label{uni bulk}
    F (Q)=\frac{s^2}{27}(9a-2bs+3cs^2)\triangleq f(s),~\text{if $Q$ is uniaxial}.
 \end{equation}

Additionally, when choosing $s=s_\pm$, where:
\begin{equation}\label{splus}
 s_-=0 \quad\text{ and } \quad s_+=\frac{b+\sqrt{b^2-24ac}}{4c},
\end{equation}
we ascertain that $f'(s)=0$, which indicates that $f(s)$ reaches local minima.

In this study, we will focus on the bistable case(this means liquid crystal materials arrange themselves into the nematic phase and the isotropic phase with equal probability ) when
\begin{equation}\label{critical coeff}
b^2=27ac,\quad\text{and}~  a,c>0.
\end{equation}
Through rescaling, we can select $a=3,b=9,c=1$.
From the physics viewpoint, the choices of these coefficients correspond to the so-called eutectic point at which the system  simultaneously tend to favor the nematic phase and the isotropic phase \cite[Section 2.3]{DeGennesProst1995}.
In this case, the local minimizers to \eqref{splus} become the  global minimizers of $F(Q)$:
\begin{equation}\label{lemma1}
F(Q)\geq 0~ \text{and  the equality holds  if and only if}~ Q\in \{0\}\cup \N,
\end{equation}
 where
\begin{equation}\label{manifold}
  \N \triangleq \left\{Q\in\Q\mid Q=s_+\(\mathbf{e}\otimes \mathbf{e}-\frac 13I_3\)~\text{for some}~\mathbf{e}\in\BS\right\},\qquad \text{with}~s_+=\sqrt{\frac {3a}c}.
\end{equation}

We will show that as the parameter $\e$ approaches zero, the limit of (\ref{9.14.24}) corresponds to the following system.
\begin{subequations} \label{3.3.1}
\begin{align}
\partial_t\mathbf{v}+(\mathbf{v}\cdot \nabla) \mathbf{v}-\Delta \mathbf{v}+\nabla p & =0 && \text { in } \Omega^{\pm}(t),\  t \in [0, T_{0}], \label{11.8.1}\\
\operatorname{div} \mathbf{v} & =0 && \text { in } \Omega^{\pm}(t),\  t \in [0, T_{0}], \\
{[2 D \mathbf{v}-p \mathbf{I}] \mathbf{n}_{\Gamma_{t}}} & =-\sigma H_{\Gamma_{t}} \mathbf{n}_{\Gamma_{t}} && \text { on } \Gamma_{t},\  t \in [0, T_{0}], \label{9.17.1}\\
{[\mathbf{v}]} & =0 && \text { on } \Gamma_{t},\  t \in [0, T_{0}],\label{9.17.23} \\
\mathbf{v}|_{\partial \Omega} & =0 && \text { on } \partial \Omega \times [0, T_{0}], \\
V_{\Gamma_{t}}-\mathbf{n}_{\Gamma_{t}} \cdot \mathbf{v}|_{\Gamma_{t}} & =H_{\Gamma_{t}} && \text { on } \Gamma_{t},\  t \in [0, T_{0}] \label{11.8.2}.
\end{align}
\end{subequations}
The free boundary $\Gamma_t$ is determined by an evolution through mean curvature flow. And we define $\Gamma=\bigcup_{t\in [0,T_0]}\Gamma_t \times \{t\}$. The domain $\Omega$ is divided by $\Gamma_t$ into two parts $\Omega^\pm(t)$ and the smooth simply-connected domain $\Omega^+(t)$ is closed by $\Gamma_t$  for each $t\in[0, T_0]$. Moreover, $D \mathbf{v}=\frac{1}{2}(\nabla\mathbf{v}+(\nabla\mathbf{v})^{\top})$ represents the stress tensor. Additionally, $V_{\Gamma_{t}}$ and $H_{\Gamma_t}$ denote the normal velocity and (mean) curvature of the interface $\Gamma_t$, while $\mathbf{n}_{\Gamma_{t}}$ is the outward normal vector to $\Omega^-(t)$. $\sigma>0$ is the surface tension coefficient and the definition is provided as follows: 
\begin{align}\label{1.4.2.17}
\sigma=\frac 2{\sqrt{3}} \int^{s_+}_{0} \sqrt{ f(\tau) } \, d\tau ,
\end{align}
 where $f(\tau)$ is defined in \eqref{uni bulk}. Finally, in \eqref{9.17.1} and \eqref{9.17.23}, $[h]$ represents the jump of $h$ across $\Gamma_t$, the definition of which is provided below:
$$
[h](x, t)=\lim _{d \rightarrow 0+}[h\left(x+\mathbf{n}_{\Gamma_{t}}(x) d\right)-h\left(x-\mathbf{n}_{\Gamma_{t}}(x) d\right)].
$$

The existence of a generalized solution for the problem \eqref{3.3.1} has been studied in \cite{liu Chun}. However, in the calculations presented below, we need regularity assumption for the velocity field:
\begin{align}\label{9.18.1}
\mathbf{v} \in W^{1, \infty}\left([0, T] ; W^{1, \infty}(\Omega)\right) \cap C_t^1 C_x^0(\bar{\Omega} \times[0, T] \backslash \Gamma) \cap C_t^0 C_x^2(\bar{\Omega} \times[0, T] \backslash \Gamma),
\end{align}
which has been established and proven in \cite{Abels-Moser, Sebastian Hensel}.

In the process of studying the sharp interface limit, the relative entropy method  helps avoid the complexity of the construction of approximate solutions, so we adopt it in our paper.
Inspired by \cite{Fischer J1,fischer2020convergence,MR3353807}, the relative energy for the models \eqref{9.14.24} and \eqref{3.3.1} is defined in the following manner.
\begin{align}\label{entropy}
E \left[\mathbf{v}_{\varepsilon}, Q_{\varepsilon} \mid \mathbf{v}, \chi\right](t)\triangleq  &\int_\Omega \frac{1}{2}|\mathbf{v}_{\varepsilon}-\mathbf{v}|^2(\cdot,t)\;\mathrm{d}x+ \int_\O \(\frac{\varepsilon}{2}\left|\nabla Q_\ve(\cdot,t)\right|^2+\frac{1}{\varepsilon} {F_\ve(Q_\ve(\cdot,t))}- \bxi\cdot\nabla \psi_\ve(\cdot,t) \)\,\mathrm{d} x,
\end{align}
where 
\begin{align}
F_\ve(q)& \triangleq F(q)+\ve^{3}.\label{new bulk}
\end{align}
The vector field $\boldsymbol{\xi}$ is an appropriate extension of the unit normal vector field $\mathbf{n}_{\Gamma_{t}}$, and its specific definition and additional properties will be detailed in Section 2.

We also introduce the measure for the difference in the phase indicators, which is defined as follows:
\begin{align} \label{area diff}
E_{\mathrm{vol}}\left[Q_{\varepsilon} \mid \chi\right](t) \triangleq \int_{\Omega}(\sigma \chi-\psi_{\varepsilon})(\cdot,t) \vartheta\(d_\Gamma(\cdot,t)\) \,\mathrm{d} x,
\end{align}
where 
\begin{align*}
\chi=\chi_{\Omega^-}, \quad \psi_\ve(x,t)&\triangleq d^F_\ve\circ Q_\ve(x,t),\quad\text{and}~ d^F_\ve(q) \triangleq (\phi_{\ve} *d^F)(q),~\forall q\in\Q.
\end{align*}
Additionally, $\vartheta$ represents a truncation of the signed distance function $d_\Gamma$, $\phi_{\ve}$ denotes a family of mollifiers, and $d^F$ corresponds to  a quasi-distance function. Detailed explanations and  properties  of these elements will be provided in the Section 2.

With these preparations in place, we are now ready to present the main theorems.

\begin{theorem} \label{thm1.1}
Assume that the system of equations (\ref{9.14.24}) admits a global weak solution $(\mathbf{v}_{\varepsilon},Q_{\varepsilon})$
on a time interval $[0, T_1]$ with $T_{1} \in(0, \infty)$ in the sense of Definition \ref{weak solutions}, and $(\mathbf{v},\Gamma)$ is a strong solution to the sharp interface
limit model (\ref{3.3.1}) on $[0, T_0]\ (T_0\leq T_1)$ in the sense of Definition \ref{strong solutions}.
Also, the initial data satisfy the assumption 
\begin{equation}\label{initial}
 {E \left[\mathbf{v}_{\varepsilon}, Q_{\varepsilon} \mid \mathbf{v}, \chi\right](0)+E_{\mathrm{vol}}\left[Q_{\varepsilon} \mid \chi\right] (0)\leq C_0\ve}
\end{equation}
for some constant $C_0$ that does not depend on $\ve$,  where $E \left[\mathbf{v}_{\varepsilon}, Q_{\varepsilon} \mid \mathbf{v}, \chi\right]$ and $E_{\mathrm{vol}}\left[Q_{\varepsilon} \mid \chi\right]$ are defined in\eqref{entropy} and \eqref{area diff} respectively.
Then there exist positive constants $C=C(\mathbf{v},\Gamma, T_0)$ and $\varepsilon_0\in(0,1]$,
such that the following estimate
\begin{align*}
E \left[\mathbf{v}_{\varepsilon}, Q_{\varepsilon} \mid \mathbf{v}, \chi\right](T)+E_{\mathrm{vol}}\left[Q_{\varepsilon} \mid \chi\right](T)+\frac{1}{2} \int_0^T \!\int_{\Omega} |\nabla \mathbf{v}_{\varepsilon}-\nabla \mathbf{v}|^2 \mathrm{d} x \mathrm{d} t \le  C \e
\end{align*}
holds true for any $\varepsilon \in(0,\varepsilon_0)$ and almost every $T\in(0,T_0)$.
Furthermore,
\begin{align*}
&\frac 1{4\ve} \int_0^T \!\int_{\Omega} \left| \ve \Big(\partial_{t} Q_{\varepsilon}+(\mathbf{v}_{\varepsilon} \cdot \nabla) Q_{\varepsilon}\Big) -(\operatorname{div} \bxi) D d^F_\ve(Q_\ve)  \right|^2\mathrm{d} x \mathrm{d} t \notag\\
+&\frac {\ve}{4}  \int_0^T \!\int_{\Omega} |\partial_{t} Q_{\varepsilon}+(\mathbf{v}_{\varepsilon} \cdot \nabla) Q_{\varepsilon}+(\HH \cdot\nabla) Q_\ve|^2\,\mathrm{d} x \mathrm{d} t \le  C \e.
\end{align*}
\end{theorem}

At this point, we digress to mention that the sharp interface limit for the scalar case has been widely studied, as seen in references like \cite{Fei, Liu Y N, fischer2020convergence, Sebastian Hensel, Jiang}. We arrive at this theorem under the assumption that the initial conditions satisfy \eqref{initial}, rather than
\begin{equation} \label{24.1.30.1}
 {E \left[\mathbf{v}_{\varepsilon}, c_{\varepsilon} \mid \mathbf{v}, \chi\right](0)+E_{\mathrm{vol}}\left[c_{\varepsilon} \mid \chi\right] (0)\leq C_0\ve^2}.
\end{equation}
in scalar form. In the scalar case, $Q_{\varepsilon}$ is replaced by $c_{\varepsilon}$.  For the scalar form $c_{\varepsilon}$, the initial value satisfying \eqref{24.1.30.1} is direct. However, for the tensor $Q_{\varepsilon}$, it is not clear whether an initial value satisfying  the same condition can be achieved. In fact, following the construction of the initial condition $Q_{0, \varepsilon}$ presented in \cite[Proposition 2.3]{liu}, we can prove \eqref{initial}. We will restate its construction and provide a detailed proof of it in Proposition \ref{prop 2.9}.

\begin{theorem}\label{main thm}
Let's consider that the initial value still satisfies \eqref{initial} and $\Omega^+(t)$ is a smooth simply-connected domain. Then there exists a sequence of $\ve_k$ that tends to 0 as $k$ approaches $\infty$, such that
\begin{equation}\label{strong global of Q}
  Q_{\ve_k}\xrightarrow{k\to\infty }   Q=s_\pm\left(\mathbf{u} \otimes \mathbf{u}-\tfrac{1}{3} I_3\right),~\text{strongly in}~  C([0,T];L^2_{loc}(\Omega^\pm(t)))
\end{equation}
holds true for some $T \le T_0$,  where $s_\pm$ are defined in \eqref{splus}  and
\begin{equation}\label{reg limit}
\mathbf{u}\in H^1( \Omega^+_T;\BS), \quad \Omega^+_T \triangleq\bigcup_{t\in (0,T)}\(\Omega^+(t)\times \{t\}\).
\end{equation}
Furthermore, $\uu$ represents a harmonic map heat flow into $\BS$ with homogenous Neumann boundary conditions, which satisfies
\begin{equation}\label{weak harmonic}
 \int_0^T \!\int_{\Omega^+(t)}  \p_t \mathbf{u}\wedge \mathbf{u}\cdot \boldsymbol{\varphi}\, \mathrm{d} x \mathrm{d} t+\sum_{j=1}^3 \int_0^T \!\int_{\Omega^+(t)} v_j(\p_j  \mathbf{u}\wedge \mathbf{u})\cdot \boldsymbol{\varphi}\, \mathrm{d} x \mathrm{d} t  =-\sum_{j=1}^3 \int_0^T \!\int_{\Omega^+(t)} \p_j  \mathbf{u}  \wedge\mathbf{u} \cdot  \p_j \boldsymbol{\varphi}\, \mathrm{d} x \mathrm{d} t
\end{equation}
for any $\boldsymbol{\varphi}\in C^1(\overline{\Omega}\times [0,T];\R^3)$, where $\wedge$ means the wedge product in $\R^3$.
\end{theorem}


Theorem \ref{main thm} can be viewed as a positive answer to the Keller-Rubinstein-Sternberg problem, as discussed in \cite{Rubinstein}, with the inclusion of the fluid system. The asymptotic behavior as $\e \rightarrow 0$ for  the energy-minimizing static solutions in the time-independent case of the Keller-Rubinstein-Sternberg problem was analyzed in \cite{lin} and its continuation \cite{lin 1}. By employing the Ginzburg–Landau approximation approach, \cite{wang} explored  the asymptotic limit of the solutions of the Q-tensor flow. A similar result was also determined by \cite{fei2015dynamics,MR4059996} through the utilization of the matched asymptotic expansion method and the spectral condition of a linearized operator. \cite{MR3910590,MR4076075} investigated a model problem involving transitions between nematic and isotropic phases with highly disparate elastic constants. Very recently, \cite{lin 2} studied the isotropic-nematic phase transition of liquid crystals. 

This article uses the relative entropy method and convergence method to prove the solution to the Beris-Edward  system converges to the one to the system of the Navier–Stokes equations coupled with the harmonic map flow as the $\e$ tends to zero. 
\cite{liu} focused on the model of nematic-isotropic phase transitions and proved a convergence result. Compared with the reference \cite{liu}, we extended our consideration to the coupled system with fluid dynamics. Besides the relative entropy estimates obtained in their paper, we also obtain estimates for the bulk error. Convergence rates for both aspects are also provided under well-prepared initial data.
Analytically, the uniform estimate for $Q_\varepsilon$ is crucial. Specifically, obtaining the uniform estimate allows us to avoid certain technical difficulties in the proof of Lemma \ref{basic control est}. However, due to the presence of the fluid, the proof of the uniform bound estimate becomes trickly different. Consequently, our approach in dealing with this  is guided by \cite{Guillen-Gonzalez}.
Furthermore, because of the influence of the fluid, a suitable extension for the velocity field $\vv$ is required.
Finally, the capillary term $\operatorname{div} (\nabla Q_{\varepsilon} \odot \nabla Q_{\varepsilon})$ in \eqref{9.6.4}is  a well-known challenge in the study of sharp interface limits. This is due to the strong layer structure of $Q_{\varepsilon}$ appearing in the interface region, which is singular and unbounded. And we employ the energy stress tensor $T_\ve$ in \eqref{24.1.22.1} to overcome this difficulty.

The structure of this paper is as follows. Section 2 provides definitions and notations commonly used, while Section 3 derives corresponding estimates for the relative energy. Section 4 provides estimates for the bulk error. Finally, in Section 5, the Theorem \ref{thm1.1} is proven using the previous  estimates. Additionally, we employ convergence method to complete the proofs of Theorem \ref{main thm}.
 \section{Preliminaries}\label{sec Modulated}

We start with the definition of weak solutions to the Beris-Edward  system \eqref{9.14.24}. We need the following function spaces:
\begin{align*}
& L_\sigma^2(\Omega) =\left\{\vv \in L^2\left(\Omega ; \mathbb{R}^d\right), \operatorname{div} \vv=0, \gamma(\vv)=0\right\}, \\
& H_{0, \sigma}^1(\Omega) =\left\{\vv \in H_0^1\left(\Omega ; \mathbb{R}^d\right), \operatorname{div} \vv=0\right\}.
\end{align*}
Here $\gamma(\vv)=\vv \cdot \nn \in H^{-\frac{1}{2}}(\partial \Omega)$ is defined in a generalized trace sense,  where $\nn$ is the normal vector of $\partial \Omega$. Furthermore, if $X$ is a Banach space and $T>0$, then $B C_w([0, T] ; X)$ consists of all bounded functions $f:[0, T] \rightarrow X$ that are weakly continuous.

\begin{definition} \label{weak solutions}
$(\mathbf{v}_{\varepsilon},Q_{\varepsilon})$ is called a weak solution of the system of equations \eqref{9.14.24},  if for all $T\in(0,T_1)$, the pair $(\mathbf{v}_{\varepsilon},Q_{\varepsilon})$ satisfies the following  requirements:\\
i) It holds
\begin{align*}
& \mathbf{v}_{\varepsilon} \in B C_w\left([0, T] ; L_\sigma^2(\Omega)\right) \cap L^2\left(0, T ; H_{0, \sigma}^1(\Omega)\right), \\
& Q_{\varepsilon} \in B C_w\left([0, T] ; H^1\left(\Omega ; \Q\right)\right) \cap L^2\left(0, T ; H^2\left(\Omega ; \Q\right)\right).
\end{align*}
ii) For any $\boldsymbol{\eta} \in C^1\left([0, T] ; H_{0, \sigma}^1(\Omega) \cap W^{1, \infty}\left(\Omega ; \mathbb{R}^d\right)\right)$ and $\Psi \in C^1\left([0, T] ; H^1\left(\Omega ; \Q\right)\right)$, it holds that
\begin{align*} 
&\int_{\Omega} \boldsymbol{\eta} \cdot \mathbf{v}_\varepsilon \;\mathrm{d} x\Big|_{t=0}^{t=T}+ \int_0^T \!\int_{\Omega} -\partial_t\boldsymbol{\eta} \cdot \mathbf{v}_\varepsilon -\mathbf{v}_\varepsilon \otimes \mathbf{v}_\varepsilon :\nabla \boldsymbol{\eta} \;\mathrm{d} x \mathrm{d} t +\int_0^T \!\int_{\Omega}  \nabla \mathbf{v}_\varepsilon: \nabla \boldsymbol{\eta} \;\mathrm{d} x \mathrm{d} t \\
=&\varepsilon \int_0^T \!\int_{\Omega} \nabla Q_{\varepsilon} \odot \nabla Q_{\varepsilon}:\nabla \boldsymbol{\eta}\;\mathrm{d} x \mathrm{d} t
\end{align*}
and
\begin{align*}
&\int_{\Omega} Q_{\varepsilon}: \Psi \mathrm{~d} x \Big|_{t=0}^{t=T} -\int_0^T \!\int_{\Omega} Q_{\varepsilon}: \partial_t \Psi \, \mathrm{d} x \mathrm{d} t+\int_0^T \!\int_{\Omega} (\mathbf{v}_\varepsilon \cdot \nabla) Q_{\varepsilon}: \Psi \, \mathrm{d} x \mathrm{d} t \\
& =\int_0^T \!\int_{\Omega} (\Delta Q_{\varepsilon}-\frac{1}{\varepsilon^{2}} D F (Q_{\varepsilon})): \Psi \, \mathrm{d} x \mathrm{d} t .
\end{align*}
iii) For almost every $T$ the following energy inequality is valid:
\begin{eqnarray} \label{energy inequality}
\begin{split}
&\frac{1}{2} \int_{\Omega} |\mathbf{v}_\varepsilon(T)|^2\mathrm{d} x +\int_{\Omega} \left[\frac{\varepsilon}{2}|\nabla Q_{\varepsilon}|^2+ \frac{1}{\varepsilon}F_\e (Q_{\varepsilon})\right](T)\, \mathrm{d} x\\
&+\int_0^T \!\int_{\Omega} |\nabla \mathbf{v}_\varepsilon|^2 \mathrm{d} x \, \mathrm{d} t+\int_0^T \!\int_{\Omega} \frac{1}{\varepsilon} \Big(\varepsilon \Delta Q_{\varepsilon}-\frac{1}{\varepsilon} D F_\e (Q_{\varepsilon})\Big)^{2} \mathrm{d} x \, \mathrm{d} t\\
\le&\frac{1}{2} \int_{\Omega} |\mathbf{v}_\varepsilon(0)|^2\mathrm{d} x+\int_{\Omega} \left[\frac{\varepsilon}{2}|\nabla Q_{\varepsilon}|^2+ \frac{1}{\varepsilon}F_\e (Q_{\varepsilon})\right](0)\, \mathrm{d} x.
\end{split}
\end{eqnarray}
\end{definition}
\cite{Abels-Dolzmann-Liu} establishes the existence of solutions under more general Dirichlet–Neumann boundary conditions. The existence here only uses the specific case presented in Remark 1.4 of their article.

After this, we provide the definition of strong solutions for the sharp interface limit model.

\begin{definition} \label{strong solutions}
$(\mathbf{v},\Gamma)$ is called a strong solution for the system of equations (\ref{3.3.1}), it satisfies the following conditions:\\
i) It holds
$$
\begin{cases}
\mathbf{v} \in H^1(0,T;L^2(\Omega)) \cap L^2(0,T;H^1(\Omega)),\\
\mathbf{v} \in W^{1, \infty}\left([0, T] ; W^{1, \infty}(\Omega)\right) \cap C_t^1 C_x^0(\bar{\Omega} \times[0, T] \backslash \Gamma) \cap C_t^0 C_x^2(\bar{\Omega} \times[0, T] \backslash \Gamma).
\end{cases}
$$
ii) The velocity field $\mathbf{v}$ simultaneously fulfills the conditions $\div \mathbf{v}(\cdot, t)=0$ and the momentum balance equation in the distributional sense. Specifically, for all $\boldsymbol{\eta} \in C_c^{\infty}(\Omega \times[0, T])$ satisfying $\div \boldsymbol{\eta}=0$, one has
\begin{small}
\begin{align*}
& \int_{\Omega} \mathbf{v} \cdot \boldsymbol{\eta} \,\mathrm{d} x \Big|_{t=0}^{t=T} \\
= &\int_0^T \int_{\Omega} \mathbf{v} \cdot \partial_t \boldsymbol{\eta}  \,\mathrm{d} x  \,\mathrm{d} t -\int_0^T (\mathbf{v}\cdot \nabla) \mathbf{v} \cdot \boldsymbol{\eta} \,\mathrm{d} x  \,\mathrm{d} t -\int_0^T \int_{\Omega} \nabla \mathbf{v}: \nabla \boldsymbol{\eta}  \,\mathrm{d} x  \,\mathrm{d} t-\sigma \int_0^T \int_{\Gamma_t} H_{\Gamma_{t}} \mathbf{n}_{\Gamma_{t}} \cdot \boldsymbol{\eta} \,\mathrm{d} \mathcal{H}^2\,\mathrm{d} t
\end{align*}
\end{small}
holds true for almost every $T \in(0, T_0)$.

For this definition, one can also refer to \cite[Definition 4]{Sebastian Hensel} for more details. And the existence of the solution can be found in \cite{Abels-Moser, Sebastian Hensel}.
\end{definition}

We are going to introduce the concepts that appeared in relative entropy \eqref{entropy} and bulk error \eqref{area diff} in the following. The signed distance function is defined as follows:
$$
d_{\Gamma}(x, t)\triangleq \operatorname{sdist}(x,\Gamma_{t})=
\begin{cases}
\operatorname{dist}\left(\Omega^{-}(t), x\right) & \text { if } x \notin \Omega^{-}(t), \\
-\operatorname{dist}\left(\Omega^{+}(t), x\right) & \text { if } x \in \Omega^{-}(t).\end{cases}
$$
And we choose a suitable $\delta>0$ which satisfies the condition that the distance between the interface $\Gamma_t$ and $\partial \Omega$ is at least $3\delta$.
Furthermore, 
$$
\Gamma_t(3\delta)\triangleq \{y\in \Omega: \hbox{dist}(y, \Gamma_t)<3\delta\} \text{ and } \Gamma(3\delta)=\bigcup_{t\in (0, T_0)}\Gamma_t(3\delta)\times\{t\}.
$$

For every point $x \in \Gamma_{t}$, there exists a local diffeomorphisms $X_{0}: \mathbb{T}^{2} \times (0, T_{0}) \rightarrow \Gamma_{t}$. For any $x=X_{0}(s, t) \in \Gamma_{t}$, we denote
$$
\mathbf{n}_{\Gamma_{t}}(x, t)\triangleq \mathbf{n}(s, t).
$$
Then we have 
\begin{eqnarray} \label{9.19.1}
\nabla d_{\Gamma}(x, t)=\mathbf{n}_{\Gamma_{t}}\left(P_{\Gamma_{t}}(x), t\right),\quad \partial_{t} d_{\Gamma}(x, t)=-V_{\Gamma_{t}}\left(P_{\Gamma_{t}}(x), t\right),\quad \Delta d_{\Gamma}(p, t)=-H_{\Gamma_{t}}(p, t)
\end{eqnarray}
holds for $\forall (x,t)\in \Gamma(3 \delta)$ and $(p,t)\in \Gamma$, where $P_{\Gamma_{t}}(x)$ is the orthogonal projection (cf. \cite[Section 4.1]{Chen X}).

Next, we extend the mean curvature vector $H_{\Gamma_t}$ on $\Gamma_{t}$ to the entire region $\Omega$.
\begin{definition}
The extended mean curvature vector, denoted as $\mathbf{H}(x,t)$, is defined by
\begin{align} \label{2.16.1}
\mathbf{H}(x,t) \triangleq  H_{\Gamma_{t}}(P_{\Gamma_{t}}(x),t)\mathbf{n}_{\Gamma_{t}}(x,t)\zeta(x,t),
\end{align}
where $x=P_{\Gamma_{t}}(x)+d_{\Gamma}(x,t) \mathbf{n}_{\Gamma_{t}}(x,t)$ and $\zeta$ is a cut-off function taking the form
\begin{align*}
\zeta(\cdot,t) \in C^\infty_c(\Gamma_t(2\delta))\quad  \text{ and } \quad  \zeta(\cdot,t)=1 \quad \text{ in } \Gamma_t(\delta).
\end{align*}
\end{definition}
It can be directly observed from the definition that there exists $C=C(\Gamma_0)$ such that
\begin{align}
\left|\HH\right|+\left|\nabla \HH\right| \leq C  \label{7.25.1}.
\end{align}

The constant extension of $\mathbf{v}$ away from $\Gamma_{t}$ is defined as follows:
\begin{align} \label{2.23.2}
\mathbf{\tilde{v}}(x,t) \triangleq \mathbf{v}(P_{\Gamma_{t}}(x),t)  \quad \text{ in }\ \Gamma_{t}(3\delta).
\end{align}
It follows from \eqref{9.17.23} and \eqref{9.18.1} that
\begin{align} \label{regular tilde v}
\tilde{\vv}\in C_t^0C_x^2(\Gamma(3\delta))\cap C_t^1C_x^0(\Gamma(3\delta)).
\end{align}
Based on Lipschitz condition in \eqref{9.18.1}, we know that there exists a non-negative bounded function $\omega(t)$ such that 
\begin{align} \label{lipschitz uu}
|\vv(x,t)-\tilde{\vv}(x,t)|\leq \omega(t) |\dd_\Gamma(x,t)|
\end{align}
holds.

The non-negativity of relative entropy largely depends on the choice of $\boldsymbol{\xi}$. $\boldsymbol{\xi}$ is defined as an extension of the unit inner normal vector to $\Omega^+(t)$ and is given by:
\begin{align} \label{2.27.1}
\boldsymbol{\xi}(x, t)\triangleq \phi\left(\frac{\mathrm{d}_{\Gamma}(x, t)}{\delta}\right) \nabla \mathrm{d}_{\Gamma}(x, t),
\end{align}
where $\phi(x) \geq 0$ is an even, nonnegative cutoff function  defined on $\mathbb{R}$. It monotonically decreases on $[0, 1]$ and satisfies 
$$
\begin{cases}\phi(x)>0\quad \text { for } & |x|<1, \\ \phi(x)=0\quad \text { for } & |x| \geq 1, \\ 1-4 x^2 \leq \phi(x) \leq 1-\frac{1}{2} x^2 & \text { for }|x| \leq 1 / 2.\end{cases}
$$
The last condition guarantees $\phi'(x)\sim O(x)$ in the interval $[-\frac{1}{2}, \frac{1}{2}]$. Therefore, for some constant $C$, it is evident that we have 
\begin{align}
&\boldsymbol{\xi} \in C^{0,1}\left([0, T] ; L^{\infty}(\Omega)\right) \cap L^{\infty} \left([0, T] ; C_c^{1,1}(\Omega)\right), \quad  \left\|(\partial_t \boldsymbol{\xi}, \nabla^2 \boldsymbol{\xi})\right\|_{L^{\infty}(\Omega \times(0, T))} \leq C , \label{2.28.1}
\end{align}
and
\begin{align}
&|\boldsymbol{\xi}|  \leq 1-C \min \left\{d_{\Gamma}^2, 1\right\} & & \text { a.e. on } \Omega \times[0, T], \notag\\
&\boldsymbol{\xi}  = \mathbf{n}_{\Gamma_{t}} \text { and } \div \boldsymbol{\xi}=- H_{\Gamma_{t}} & & \text { on } \Gamma_{t} ,\label{2.18.2}\\
&(\boldsymbol{\xi} \cdot \nabla)\mathbf{H}=0, \quad (\boldsymbol{\xi} \cdot \nabla)\mathbf{\tilde{v}}=0 \quad &&\text{ in } \Omega \label{2.23.1}.
\end{align}

In fact,  it holds that 
\begin{align} \label{trans d}
\p_t d_\Gamma(x,t)+(\HH+\tilde{\vv})\cdot \nabla d_\Gamma(x,t) =0 \qquad \text{in } \Gamma_t(\delta). 
\end{align}
Then, according to \eqref{9.19.1} and \eqref{2.16.1}, we can derive(cf. \cite{Sebastian Hensel}) 
\begin{align}
&|\mathbf{H} \cdot \boldsymbol{\xi}+\div \boldsymbol{\xi}|  \leq C \min \{d_{\Gamma}, 1\} & \text { a.e. on } \Omega \times[0, T], \label{2.20.1}\\
&\left|\partial_t \boldsymbol{\xi}+\left((\HH+\vv) \cdot \nabla\right) \bxi +\left(\nabla \HH+\nabla \tilde{\vv}\right)^{\top} \boldsymbol{\xi} \right| \leq C \min \{d_{\Gamma}, 1\} & \text { a.e. on } \Omega \times[0, T], \label{2.16.5}\\
&\left|\boldsymbol{\xi} \cdot\Big(\partial_t+\left((\HH+\vv) \cdot \nabla\right)\Big) \boldsymbol{\xi}\right| \leq C \min \left\{d_{\Gamma}^2, 1\right\} & \text { a.e. on } \Omega \times[0, T]  \label{2.16.7}.
\end{align}

Now, let us give the definitions of $\phi_{\ve}$ and $d^F$, both of which appear in the area differences \eqref{area diff}. Recall that
\begin{align} \label{psi}
\psi_\ve(x,t)&\triangleq d^F_\ve\circ Q_\ve(x,t),\quad\text{and}~ d^F_\ve(q) \triangleq (\phi_{\ve} *d^F)(q),~\forall q\in\Q .
\end{align}
$\phi_{\ve}$ is a family of smooth, non-negative mollifiers with compact support in $B_1^\Q$ (the ball of radius 1 in the space $\Q$) given by
\begin{equation}\label{psi convolu}
\phi_{\ve}(q)\triangleq \ve^{-20}\phi\(  \ve^{-4}q\).
\end{equation}
Additionally, $\phi$ is isotropy, which means that for any orthogonal matrix $M\in O(3)$ and any $q\in \Q$, it satisfies $\phi(M^\top qM)=\phi(q)$. And
\begin{equation}\label{quasidistance}
d^F(q)\triangleq \inf\left\{\int_0^1\sqrt{2 F(\gamma(t))}|\gamma'(t)|\, dt \Big|  \gamma\in C^{0,1}([0,1];\Q),\gamma(0)\in \N,\gamma(1)=q\right\}
\end{equation}
is the quasi-distance function, which was introduced by  \cite{Sternberg}  and independently by \cite{MR985992}. One  can also refer to \cite{liu}.
The definition of $\phi_{\ve}$ in \eqref{psi convolu} is motivated by technical challenges encountered in the proof of Lemma \ref{basic control est}.

Let us consider a special but crucial case: when the $Q$-tensor is uniaxial, one can obtain a specific expression for $d^F$.
\begin{lemma}\label{lemma unixial}
Let $f(s)$ be defined by \eqref{uni bulk}. If $Q$ is represented as $Q=s_0\left(\mathbf{u} \otimes \mathbf{u}-\frac{1}{3} I_3\right)$ for some $s_0\in [0,s_+]$ and a unit vector $\mathbf{u}  \in \BS$, then
\begin{equation}\label{precise diff}
 d^F(Q)= \frac 2{\sqrt{3}} \int^{s_+}_{s_0} \sqrt{ f(\tau) } \, d\tau \triangleq g(s_0).
\end{equation}
 \end{lemma}
\begin{proof}
This proof is similar to the one in \cite{Park}, so we won't provide details here.
\end{proof}

It is obvious from \eqref{psi} that
\begin{align} 
 \nabla \psi_\ve(x,t) & =  D d^F_\ve(Q_\ve) \colon \nabla Q_\ve(x,t) \qquad \text{for a.e.\ }(x,t)\in \Omega\times (0,T),\label{ADM chain rule}\\
\p_t \psi_\e(x,t) &=  D d^F_\e(Q_\e) \colon \p_t Q_\e(x,t) \qquad \text{for a.e.\ }(x,t)\in \Omega\times (0,T) \label{8.11.7}.
\end{align}

The inspiration from \eqref{ADM chain rule} leads  us to define a projection operator $\Pi_{Q_\ve}$:
\begin{equation} \label{projection1}
\Pi_{Q_\ve}  \p_i Q_\ve=\left\{
\begin{array}{rl}
\(\p_i Q_\ve:\frac{D d^F_\ve(Q_\ve)}{|D d^F_\ve(Q_\ve)|}\) \frac{D d^F_\ve(Q_\ve)}{|D d^F_\ve(Q_\ve)|},&~\text{if}~D d^F_\ve(Q_\ve)\neq 0,\\
0,&~\text{otherwise}.
\end{array}
\right.
\end{equation}
Thus, we obtain from \eqref{ADM chain rule} that
\begin{subequations}
\begin{align}
|\nabla \psi_\ve| &= |\Pi_{Q_\ve} \nabla Q_\ve| |D d^F_\ve(Q_\ve)| \qquad \qquad \text{for a.e.\ }(x,t)\in \Omega\times (0,T), \label{projectionnorm}\\
\Pi_{Q_\ve} \nabla Q_\ve&=\frac{|\nabla\psi_\ve|} {|D d^F_\ve(Q_\ve)|^2}D d^F_\ve(Q_\ve)\otimes \nn_\ve   \qquad \text{for a.e.\ }(x,t)\in \Omega\times (0,T), \label{projection}
\end{align}
\end{subequations}
where 
\begin{align}
\mathbf{n}_\varepsilon &\triangleq \frac{\nabla \psi_\varepsilon}{|\nabla \psi_\varepsilon|}
\end{align}
is the analogous normal vector.

The following lemma provides control over the gradient of the convolution $d_\ve^F$, which is crucial for closing the estimates of  the energy of relative entropy.
\begin{lemma}\label{basic control est}
For any fixed $c_0>0$, there is a corresponding  $\ve_0 >0$ such that  if $q\in\Q$ and $|q|\leq c_0$, then the following inequality holds for all $\ve \in (0, \ve_0)$:
\begin{align} \label{sharp lip d}
|D d^F_\ve(q)| \leq \sqrt{2 F_\ve(q)}.
\end{align}
\end{lemma}
\begin{proof}
The proof relies on the maximum principle established in \eqref{basic constant1}, which will be showed in the next lemma. The proof of this lemma has been established in \cite[Lemma 4.1]{liu} and will not be reiterated here. 
\end{proof}

\begin{lemma} \label{L infinity bound}
(Maximum Principle) Let the pair $(\mathbf{v}_{\varepsilon},Q_{\varepsilon})$ be a global weak solution of the system of equations (\ref{9.14.24}) in the sense of Definition \ref{weak solutions} on a time interval $[0, T_1]$. Choose $c_0 > 0$ to be sufficiently large, depending only on the coefficients $(a, b, c)$ of the function $F(Q)$ and $\|Q_{0, \varepsilon}\|_{\mathrm{L}^{\infty}(\Omega)}$. More precisely, $c_0$ is independent of time and satisfies
$$
c_0^2 \geq max\{ \frac{b^2}{c^2}-\frac{2 a}{c}, \|Q_{0, \varepsilon}\|_{\mathrm{L}^{\infty}(\Omega)}^2\},
$$
then the uniform bound of $Q_\ve$ is given by
$$
\|Q_\ve\|_{L^\infty(\Omega\times(0,T_1))} \leq c_0 .
$$
\end{lemma}
\begin{proof}
The proof is similar to that of \cite[Theorem 3]{Guillen-Gonzalez}. For the sake of completeness in this paper, we provide the proof here. Upon taking the inner product of \eqref{10.19.1} with $Q_{\varepsilon}$, we obtain:
\begin{align} \label{54}
\frac{1}{2} \partial_t\left(|Q_{\varepsilon}|^2\right)+\mathbf{v}_{\varepsilon} \cdot \nabla\left(\frac{|Q_{\varepsilon}|^2}{2}\right)- \frac{1}{2} \Delta\left(|Q_{\varepsilon}|^2\right)+ \frac{1}{2} |\nabla Q_{\varepsilon}|^2+\frac{1}{\e^2} D F(Q_{\varepsilon}): Q_{\varepsilon}=0,
\end{align}
which implies
\begin{align} \label{55}
\partial_t\left(|Q_{\varepsilon}|^2-c_0^2 \right)+\mathbf{v}_{\varepsilon} \cdot \nabla\left(|Q_{\varepsilon}|^2-c_0^2 \right)-  \Delta\left(|Q_{\varepsilon}|^2-c_0^2 \right) +\frac{2}{\e^2} D F(Q_{\varepsilon}): Q_{\varepsilon}=0.
\end{align}
Testing \eqref{55} by $(|Q_{\varepsilon}|^2-c_0^2)_+ $ and integrating in $\Omega$, we deduce that
\begin{align} \label{56}
\frac{d}{d t} \big\|(|Q_{\varepsilon}|^2-c_0^2)_+ \big\|_{L^2(\Omega)}^2+ \big\|\nabla (|Q_{\varepsilon}|^2-c_0^2)_+ \big\|_{L^2(\Omega)}^2 +\frac{2}{\e^2} \int_{\Omega}(D F(Q_{\varepsilon}): Q_{\varepsilon}) (|Q_{\varepsilon}|^2-c_0^2)_+ \;\mathrm{d}x \leq 0 .
\end{align}
From definition of $D F(Q)$, we know that
$$
D F(Q_{\varepsilon}): Q_{\varepsilon}=a|Q_{\varepsilon}|^2-b \,tr Q_{\varepsilon}^3 +c|Q_{\varepsilon}|^4.
$$
It follows from Young's inequality that
$$
b \,tr Q_{\varepsilon}^3 \leq \frac{c}{2}|Q_{\varepsilon}|^4+\frac{b^2}{2 c}|Q_{\varepsilon}|^2,
$$
so there holds
$$
D F(Q_{\varepsilon}): Q_{\varepsilon} \geq \frac{c}{2}|Q_{\varepsilon}|^4+\left(a-\frac{b^2}{2 c}\right)|Q_{\varepsilon}|^2=\frac{c}{2}|Q_{\varepsilon}|^2\left(|Q_{\varepsilon}|^2-\mu^2 \right),
$$
where $\mu^2=\frac{b^2}{c^2}-\frac{2 a}{c}$. If $|Q_{\varepsilon}|^2 \le \mu^2$, then the desired estimate is obtained. If not, we have:
$$
(D F(Q_{\varepsilon}): Q_{\varepsilon}) (|Q_{\varepsilon}|^2-c_0^2)_+ \geq \frac{c}{2}|Q_{\varepsilon}|^2\left(|Q_{\varepsilon}|^2-\mu^2 \right)(|Q_{\varepsilon}|^2-c_0^2)_+ \geq 0 .
$$
In this case, according to \eqref{56}, one can conclude that
$$
\frac{d}{d t} \big\|(|Q_{\varepsilon}|^2-c_0^2)_+ \big\|_{L^2(\Omega)}^2+ \big\|\nabla (|Q_{\varepsilon}|^2-c_0^2)_+ \big\|_{L^2(\Omega)}^2 \leq 0.
$$
Hence, the weak maximum principle yields that the maximum must be attained on the parabolic boundary $\(\p\O\times (0,T)\) \cup \(\Omega\times \{0\}\)$, that is
$$
\big\|(|Q_{\varepsilon}|^2-c_0^2)_+ \big\|_{L^2(\Omega)}^2 \leq \big\| (|Q_{0, \varepsilon}|^2-c_0^2)_+ \big\|_{L^2(\Omega)}^2=0.
$$
Therefore, $\|Q_{\varepsilon}(t)\|_{\mathrm{L}^{\infty}(\Omega)} \leq c_0$ is deduced.
\end{proof}

Finally, $\vartheta$ in \eqref{area diff} can be  formulated as a smooth asymmetric  truncation of the signed distance function, which takes the following form:
\begin{align} \label{24.1.15.20}
\vartheta(r)=\left\{
\begin{array}{rll}
-\delta & \text { as } & r \geq  \delta, \\
-r & \text { as } & -\delta \leq r \leq  \delta, \\
\delta & \text { as } & r \leq - \delta .
\end{array}\right.
\end{align}
It is evident from the construction and \eqref{trans d} that $\vartheta$ satisfies regularities
\begin{align} \label{9.17.4}
\vartheta \in C^{0,1}\left([0, T] ; L^{\infty}(\bar{\Omega})\right) \cap L^{\infty} \left([0, T] ; C^{0,1}(\bar{\Omega})\right),\quad \left\|\left(\partial_t \vartheta, \nabla \vartheta\right)\right\|_{L^{\infty}(\Omega \times(0, T))} \leq C, 
\end{align}
coercivity and consistency
\begin{align} \label{9.17.5}
\tilde{c} \min \{d_{\Gamma}, 1\} \leq |\vartheta| \leq C \min \{d_{\Gamma}, 1\} \quad \text { is fulfilled in } \Omega \times[0, T]
\end{align} 
and transportability property
\begin{align} \label{9.17.6}
\left|\partial_t \vartheta+ \((\HH+\tilde{\vv}) \cdot \nabla\) \vartheta\right| \leq C \min \{d_{\Gamma}, 1\} \quad \text { a.e. in } \Omega \times[0, T].
\end{align}

In the following, we aim to discuss some properties of relative entropy. To accomplish this, we define analogous mean curvature in the phase field:
\begin{align}
 \HH_\ve(x,t) \triangleq -\left(\varepsilon \Delta Q_\ve-\frac{D F(Q_\ve)}{\varepsilon} \right):\frac{\nabla Q_\ve}{\left|\nabla Q_\ve\right|}&=-\varepsilon \left(\partial_{t} Q_{\varepsilon}+(\mathbf{v}_{\varepsilon} \cdot \nabla) Q_{\varepsilon}\right):\frac{\nabla Q_\ve}{\left|\nabla Q_\ve\right|} \label{mean curvature app}.
 \end{align}

Now we present some properties of the relative energy that will be frequently employed in this paper.
\begin{lemma} \label{prop2.1}
(\cite[Lemma 4.2]{liu}) For every $T_* \in [0, T_0)$, there exists a constant $C$, such that for all $T \in [0, T_*]$, the following inequalities hold:
\begin{align} 
&\qquad \qquad \qquad \qquad \qquad \frac{1}{2} \int_\Omega |\mathbf{v}_{\varepsilon}(\cdot,t)-\mathbf{v}(\cdot,t)|^2 \;\mathrm{d}x \leq  E \left[\mathbf{v}_{\varepsilon}, Q_{\varepsilon} \mid \mathbf{v}, \chi\right](T),\\
& \int_{\Omega} |\mathbf{n}_{\varepsilon}-\boldsymbol{\xi}|^2\left|\nabla \psi_{\varepsilon}\right| \mathrm{d} x+\int_{\Omega} \(\frac{\varepsilon}{2} \left|\nabla Q_\ve\right|^2+\frac{1}{\ve} F_\ve(Q_\ve)-|\nabla \psi_\ve| \)\,\mathrm{d} x \leq  E \left[\mathbf{v}_{\varepsilon}, Q_{\varepsilon} \mid \mathbf{v}, \chi\right](T),\label{energy bound0}\\
&\frac12\int_{\Omega} \left(\sqrt{\varepsilon}\left|\Pi_{Q_\ve}\nabla Q_\ve\right|-\frac{1}{\sqrt{\ve}}\sqrt{2 F_\ve(Q_\ve)} \right)^{2}\,\mathrm{d} x+\frac \ve 2 \int_{\Omega}  \(  \left|\nabla Q_\ve-\Pi_{Q_\ve}\nabla Q_\ve\right|^2   \)\,\mathrm{d} x\leq  E \left[\mathbf{v}_{\varepsilon}, Q_{\varepsilon} \mid \mathbf{v}, \chi\right](T),\label{energy bound2}\\
&\int_{\Omega} \left(\sqrt{\varepsilon}\left|\nabla Q_\ve\right|-\frac{1}{\sqrt{\varepsilon}} |D d^F_\ve(Q_\ve)| \right)^{2} \,\mathrm{d} x  + \int_{\Omega}\left(\sqrt{\varepsilon}\left|\Pi_{Q_\ve}\nabla Q_\ve\right|-\frac1{\sqrt{\ve}} |D d^F_\ve(Q_\ve)| \right)^{2}\,\mathrm{d} x\notag\\
  &\qquad\qquad\qquad\qquad\qquad\quad+\int_{\Omega} \left(1-\bxi \cdot\nn_{\varepsilon}\right)\( {\frac{\ve}{2}}\left|\Pi_{Q_\ve}\nabla Q_\ve\right|^{2}+\left|\nabla \psi_{\varepsilon}\right|\) \,\mathrm{d} x\leq   CE \left[\mathbf{v}_{\varepsilon}, Q_{\varepsilon} \mid \mathbf{v}, \chi\right](T),\label{energy bound1} \\
&  \qquad  \int_{\Omega}  \(\frac{\varepsilon}2 \left|\nabla Q_\ve\right| ^{2}+\frac1 {\varepsilon}{ F_\ve(Q_\ve)}+|\nabla\psi_\ve|\) \min\(d_\Gamma^2,1\)\,\mathrm{d} x\leq    C E \left[\mathbf{v}_{\varepsilon}, Q_{\varepsilon} \mid \mathbf{v}, \chi\right](T). \label{energy bound3}
\end{align}
\end{lemma}

It is convenient to introduce the following lemma.
\begin{lemma} \label{lemma2.2}
For a suitably small $\lambda>0$, it holds that
\begin{equation} \label{2.19.7}
\int_{\Omega} |\sigma \chi-\psi_{\varepsilon}| |\mathbf{v}_{\varepsilon}-\mathbf{v}| \,\mathrm{d} x \le \frac{C}{\lambda}\int_{\Omega} (E_{\mathrm{vol}}\left[Q_{\varepsilon} \mid \chi\right]+ |\mathbf{v}_{\varepsilon}-\mathbf{v}|^2) \,\mathrm{d} x+\lambda \int_{\Omega}  |\nabla \mathbf{v}_{\varepsilon}-\nabla \mathbf{v}|^2 \,\mathrm{d} x.
\end{equation}
\end{lemma}
\begin{proof}
 The proof for this lemma is not provided here but can be found in (31) in \cite{Sebastian Hensel}.
\end{proof}

Finally, this section ends with the proof of the initial data \eqref{initial}. 
Let $\tilde{\zeta}(z)$ be a cut-off function defined as follows:
\begin{equation}\label{cut-off zeta}
\tilde{\zeta}(s)=1~\text{for}~|s|\leq 1/2; \quad  \tilde{\zeta}(s)=0~\text{for}~|s|\geq 1.
\end{equation}

To define $Q_{0, \varepsilon}$, we employed the construction presented in \cite{liu}.  For the sake of clarity in the proof, we redescribe it here. It is necessary to introduce the following profile $S(z)$, which is the unique increasing solution of
\begin{align} \label{travelling wave}
-S''(z)+a S(z)-\frac{b}{3}S^2(z)+\frac{2}{3} c S^3(z)=0
\end{align}
with the boundary conditions $S(-\infty)=0,\,S(+\infty)=s_+$.
Hence, we know that
\begin{equation}\label{degree of orientation}
S(z)\triangleq\frac{s_+}{2}\left(1+\tanh \left(\frac{\sqrt{a}}{2} z\right)\right),\quad z\in \R
\end{equation}
and define
\begin{align}
\tilde{S}_\ve(x)&\triangleq \tilde{\zeta}\(\frac {d_\Gamma(x,0)}{\delta}\)S\(\frac{d_\Gamma(x,0)}\ve\)+\(1-\tilde{\zeta}\(\frac {d_\Gamma(x,0)}{\delta}\)\)s_+\chi_{\Omega^+(0)} \label{cut-off initial}\\
&\triangleq S\(\frac{d_\Gamma(x,0)}\ve\)  +\hat{S}_\ve(x), \label{1.4.2.49}
\end{align}
where
\begin{equation}
\hat{S}_\ve(x) \triangleq \(1-\zeta\(\frac {d_\Gamma(x,0)}{\delta}\)\)\(s_+\chi_{\Omega^+(0)} -S\(\frac{d_\Gamma(x,0)}\ve\)\).
\end{equation}
Considering the properties of solutions to the second order ordinary differential equation \eqref{travelling wave}, we assert that 
\begin{equation}\label{hat S decay}
\|\hat{S}_\ve\|_{L^\infty(\O)}+\|\nabla \hat{S}_\ve\|_{L^\infty(\O)}\leq Ce^{-\frac C\ve}
\end{equation}
holds for some constant $C>0$ that depends only on $\Gamma_0$, cf. \cite[Section 2 and Section 3]{liu}.

As  mentioned above, in this paper, we will derive the estimation for the bulk error under a well-prepared initial condition. Therefore, we need to verify that the  given initial value satisfies some decay estimate, which will be shown in the following proposition.

\begin{proposition}\label{prop 2.9}
For  every  $\mathbf{u}_0\in H^1(\O;\BS)$,  the initial datum takes the following form:
\begin{equation}\label{sharp initial}
Q_{0, \varepsilon}(x)\triangleq\tilde{S}_\ve(x)\(\mathbf{u}_0(x)\otimes \mathbf{u}_0(x)-\frac 13 I_3\).
\end{equation}
It fulfills $Q_{0, \varepsilon}\in H^1(\O;\Q)\cap L^\infty(\O;\Q)$ and
\begin{align}\label{transition initial data}
Q_{0, \varepsilon}(x)=\left\{
\begin{array}{rl}
s_+(\mathbf{u}_0\otimes \mathbf{u}_0-\frac 13 I_3)& \quad \text{if}~ x\in \Omega^+(0)\backslash \Gamma_0(\delta),\\
S\(\frac{d_\Gamma(x,0)}\ve\)(\mathbf{u}_0\otimes \mathbf{u}_0-\frac 13 I_3)& \quad \text{if}~ x\in \Gamma_0(\delta/2),\\
0&\quad \text{if}~ x\in \Omega^-(0)\backslash \Gamma_0(\delta).
\end{array}
\right.
\end{align}
 Additionally, there exists  a constant $C_0>0$ which only depends on $\Gamma_0$ and $\|\uu_0\|_{H^1(\O)}$  such that \eqref{initial} holds true.
\end{proposition}
\begin{proof}
The proof of $E \left[\mathbf{v}_{\varepsilon}, Q_{\varepsilon} \mid \mathbf{v}, \chi\right](0) \leq C_0\ve$ is analogous to \cite[Proposition 2.3]{liu}, and we omit it here for brevity. To verify $E_{\mathrm{vol}}\left[Q_{\varepsilon} \mid \chi\right] (0)\leq C_0\ve$, we shall employ Lemma \ref{lemma unixial}. It follows from \eqref{psi} and \eqref{precise diff} that
\begin{align} \label{1.4.2.47}
\psi_\ve(x,0)=d^F(Q_{0, \varepsilon}(x))=\frac 2{\sqrt{3}} \int^{s_+}_{ \tilde{S}_\ve(x)} \sqrt{ f(\tau) } \, d\tau.
\end{align}
Recalling \eqref{cut-off zeta}, \eqref{cut-off initial} and the definition of $\sigma$ in \eqref{1.4.2.17}, we deduce that $(\sigma \chi-\psi_{\varepsilon})=0$ in $\Omega \backslash \Gamma_0(\delta)$. Therefore, it suffices to estimate it in the domain $\Omega_0^{+} \cap \Gamma_0(\delta)$, as the estimation in $\Omega_0^{-} \cap \Gamma_0(\delta)$ follows in a similar way. By using $\chi=\chi_{\Omega^-}$ and  \eqref{24.1.15.20}, we obtain
\begin{small}
\begin{align*}
&\int_{\Omega_0^{+} \cap \Gamma_0(\delta)}(\sigma \chi-\psi_{\varepsilon}) \vartheta(d_\Gamma) \,\mathrm{d} x\bigg|_{t=0} = \int_{\Omega_0^{+} \cap \Gamma_0(\delta)} \psi_{\varepsilon}(x)  d_\Gamma(x) \,\mathrm{d} x\bigg|_{t=0} \\
\stackrel{\eqref{1.4.2.47}}{=} & \frac 2{\sqrt{3}}\int_{\Omega_0^{+} \cap \Gamma_0(\delta)}\left(\int^{s_+}_{ \tilde{S}_\ve(x)} \sqrt{ f(\tau) } \, d\tau\right) d_\Gamma(x) \,\mathrm{d} x\bigg|_{t=0} \\
\stackrel{\eqref{1.4.2.49}}{=} & \frac {2\varepsilon}{\sqrt{3}} \int_{\Omega_0^{+} \cap \Gamma_0(\delta)}\left(\int^{s_+}_{ S\(\frac{d_\Gamma(x)}\ve\)} \sqrt{ f(\tau) } \, d\tau\right) \frac{d_\Gamma(x)}{\varepsilon} \,\mathrm{d} x\bigg|_{t=0}+\frac 2{\sqrt{3}}\int_{\Omega_0^{+} \cap \Gamma_0(\delta)}\left(\int^{ S\(\frac{d_\Gamma(x)}\ve\)}_{  S\(\frac{d_\Gamma(x)}\ve\)  +\hat{S}_\ve(x)} \sqrt{ f(\tau) } \, d\tau\right) d_\Gamma(x) \,\mathrm{d} x\bigg|_{t=0}\\
\stackrel{\eqref{hat S decay}}{=} & \frac {2\varepsilon}{\sqrt{3}} \int_{\Omega_0^{+} \cap \Gamma_0(\delta)}\left(\int^{s_+}_{ S\(\frac{d_\Gamma(x)}\ve\)} \sqrt{ f(\tau) } \, d\tau\right) \frac{d_\Gamma(x)}{\varepsilon} \,\mathrm{d} x\bigg|_{t=0}+O\left(e^{-C / \varepsilon}\right) \leq C \varepsilon,
\end{align*}
\end{small}
where in the last inequality, we used the following arguments. Since
\begin{equation}\label{s-poly}
f(\tau)=\frac c 9\tau^2 (\tau-s_+)^2,\quad \sqrt{f(\tau)}=\frac {\sqrt{c}}3 \tau(s_+ -\tau),~\text{for all}~ \tau\in [0,s_+]
\end{equation}
and
$$
S\(\frac{d_\Gamma(x)}\ve\)={s_+}\frac{e^{\frac{\sqrt{a}}{2}\frac{d_\Gamma(x)}\ve}}{e^{\frac{\sqrt{a}}{2}\frac{d_\Gamma(x)}\ve}+e^{-\frac{\sqrt{a}}{2}\frac{d_\Gamma(x)}\ve}},
$$
then we have
\begin{align*}
&\int^{s_+}_{ S\(\frac{d_\Gamma(x)}\ve\)} \sqrt{ f(\tau) } \, d\tau=\frac {\sqrt{c}}3 (\frac{s_+}{2} \tau^2-\frac{1}{3}\tau^3)\bigg|^{s_+}_{ S\(\frac{d_\Gamma(x)}\ve\)}\\
=&\frac {\sqrt{c}}3\frac{s_+^3}{6} \bigg[2\(\frac{e^{\frac{\sqrt{a}}{2}\frac{d_\Gamma(x)}\ve}}{e^{\frac{\sqrt{a}}{2}\frac{d_\Gamma(x)}\ve}+e^{-\frac{\sqrt{a}}{2}\frac{d_\Gamma(x)}\ve}}\)^3-3\(\frac{e^{\frac{\sqrt{a}}{2}\frac{d_\Gamma(x)}\ve}}{e^{\frac{\sqrt{a}}{2}\frac{d_\Gamma(x)}\ve}+e^{-\frac{\sqrt{a}}{2}\frac{d_\Gamma(x)}\ve}}\)^2+1\bigg]\\
=&\frac {\sqrt{c}}3\frac{s_+^3}{6} \frac{e^{-3\frac{\sqrt{a}}{2}\frac{d_\Gamma(x)}\ve}+3e^{-\frac{\sqrt{a}}{2}\frac{d_\Gamma(x)}\ve}}{\(e^{\frac{\sqrt{a}}{2}\frac{d_\Gamma(x)}\ve}+e^{-\frac{\sqrt{a}}{2}\frac{d_\Gamma(x)}\ve}\)^3}.
\end{align*}
In the view of this, we conclude
\begin{small}
\begin{align*}
\int_{\Omega_0^{+} \cap \Gamma_0(\delta)}\left(\int^{s_+}_{ S\(\frac{d_\Gamma(x)}\ve\)} \sqrt{ f(\tau) } \, d\tau\right) \frac{d_\Gamma(x)}{\varepsilon} \,\mathrm{d} x\bigg|_{t=0}=\frac {\sqrt{c}}3 \frac{s_+^3}{6} \int_{\Omega_0^{+} \cap \Gamma_0(\delta)}\frac{e^{-3\frac{\sqrt{a}}{2}\frac{d_\Gamma(x)}\ve}+3e^{-\frac{\sqrt{a}}{2}\frac{d_\Gamma(x)}\ve}}{\(e^{\frac{\sqrt{a}}{2}\frac{d_\Gamma(x)}\ve}+e^{-\frac{\sqrt{a}}{2}\frac{d_\Gamma(x)}\ve}\)^3} \frac{d_\Gamma(x)}{\varepsilon} \,\mathrm{d} x\bigg|_{t=0}\le C.
\end{align*}
\end{small}
and the proof is done.
\end{proof}
\section{Estimate of the Relative Energy} \label{sec c cali}

In this section, our goal is to establish the differential inequality for the relative energy given in \eqref{entropy}, which is represented in the following proposition.

\begin{proposition} \label{prop 3.4}
Let $ E \left[\mathbf{v}_{\varepsilon}, Q_{\varepsilon} \mid \mathbf{v}, \chi\right]$ and $E_{\mathrm{vol}}\left[Q_{\varepsilon} \mid \chi\right]$ be defined as in \eqref{entropy} and \eqref{area diff} respectively, then there exists a positive constant $C$ independent of $\e$, such that for any $T \in (0,T_0)$,
\begin{align*}
&\quad E \left[\mathbf{v}_{\varepsilon}, Q_{\varepsilon} \mid \mathbf{v}, \chi\right](T)+(1-\lambda)\int_0^T \!\int_{\Omega}  |\nabla \mathbf{v}_{\varepsilon}-\nabla \mathbf{v}|^2 \mathrm{d} x \mathrm{d} t \notag\\
&+\frac 1{\ve}(\frac{1}{2}-\lambda)\int_0^T \!\int_{\Omega}   \left| \ve \Big(\partial_{t} Q_{\varepsilon}+(\mathbf{v}_{\varepsilon} \cdot \nabla) Q_{\varepsilon}\Big) -(\operatorname{div} \bxi) D d^F_\ve(Q_\ve)  \right|^2\mathrm{d} x \mathrm{d} t \notag\\
&+\frac {\ve}{2} \int_0^T \!\int_{\Omega} |\partial_{t} Q_{\varepsilon}+(\mathbf{v}_{\varepsilon} \cdot \nabla) Q_{\varepsilon}+(\HH \cdot\nabla) Q_\ve|^2\,\mathrm{d} x \mathrm{d} t \nonumber \\
\le & E \left[\mathbf{v}_{\varepsilon}, Q_{\varepsilon} \mid \mathbf{v}, \chi\right](0)+ C \int_0^T E \left[\mathbf{v}_{\varepsilon}, Q_{\varepsilon} \mid \mathbf{v}, \chi\right] + E_{\mathrm{vol}}\left[Q_{\varepsilon} \mid \chi\right] \, \mathrm{d} t 
\end{align*}
holds true for a suitably small $\lambda>0$.
\end{proposition}

To prove it, we divide the proof into several propositions.
By adopting a similar approach as that in Lemma 4.4 in \cite{liu}, we are able to derive the following identity. 
\begin{proposition}\label{lemma exact dt relative entropy}
For any $T \in (0,T_0)$, we have
	\begin{subequations}\label{time deri 4}
		\begin{align}
&\int_{\Omega} \frac{1}{2} |\mathbf{v}_\varepsilon(t)|^2+ \left[\frac{\varepsilon}{2}|\nabla Q_{\varepsilon}|^2+ \frac{1}{\varepsilon}F_\e (Q_{\varepsilon})-\boldsymbol{\xi} \cdot \nabla \psi_{\varepsilon}\right](t)\mathrm{d} x \bigg|_{t=0}^{t=T} \notag\\
\le&-\int_0^T \!\int_{\Omega} |\nabla \mathbf{v}_\varepsilon|^2 \mathrm{d} x \, \mathrm{d} t-\frac 1{2\ve}\int_0^T \!\int_{\Omega}   \Big| \ve \Big(\partial_{t} Q_{\varepsilon}+(\mathbf{v}_{\varepsilon} \cdot \nabla) Q_{\varepsilon} \Big)  -(\operatorname{div} \bxi) D d^F_\ve(Q_\ve)  \Big|^2\mathrm{d} x \mathrm{d} t\nonumber \\
		&-\frac 1{2\ve}\int_0^T \!\int_{\Omega}   \Big| \HH_\ve-\ve|\nabla Q_\ve| \HH \Big|^2 \, \mathrm{d} x \mathrm{d} t-{\frac 1{2\ve}\int_0^T \!\int_{\Omega}     \Big(\Big|\ve \Big(\partial_{t} Q_{\varepsilon}+(\mathbf{v}_{\varepsilon} \cdot \nabla) Q_{\varepsilon} \Big)\Big|^2 -   |\HH_\ve|^2\Big)\,\mathrm{d} x \mathrm{d} t}\nonumber \\
		&-\int_0^T \!\int_{\Omega} (\operatorname{div} \boldsymbol{\xi}) \, D d^F_\e(Q_\e) \colon (\vv_\e \cdot \nabla) Q_\e\, \mathrm{d} x \mathrm{d} t +\frac 1{2\ve} \int_0^T \!\int_{\Omega}   \Big| (\operatorname{div} \bxi) |D d^F_\ve(Q_\ve)|\nn_\ve +\ve |\Pi_{Q_\ve} \nabla Q_\ve| \HH\Big|^2\,\mathrm{d} x \mathrm{d} t \label{tail1}\\
            &+ \int_0^T \!\int_{\Omega}    (\vv \cdot \nabla) \bxi\cdot\nabla \psi_\ve+ \nabla \tilde{\vv} : \bxi \otimes \nn_\ve |\nabla \psi_\ve|\,\mathrm{d} x \mathrm{d} t     \label{tail0}\\
	       &+\frac \ve{2} \int_0^T \!\int_{\Omega}   |\HH|^2\(|\nabla Q_\ve|^2-|\Pi_{Q_\ve}\nabla Q_\ve|^2\)\,\mathrm{d} x \mathrm{d} t
		 -\int_0^T \!\int_{\Omega} \nabla \HH: (\bxi-\nn_\ve) \otimes (\bxi-\nn_\ve) \left|\nabla \psi_{\varepsilon}\right|\, \mathrm{d} x \mathrm{d} t\label{tail2}\\
		&   +\int_0^T \!\int_{\Omega}   \operatorname{div} \HH  \( \frac{\ve}2 |\nabla Q_\ve|^2 +\frac{1}\ve F_\ve(Q_\ve) -|\nabla \psi_\ve| \)\,\mathrm{d} x \mathrm{d} t+\int_0^T \!\int_{\Omega}  \operatorname{div} \HH  \(1-\bxi\cdot \nn_\ve\)|\nabla\psi_\ve|\, \mathrm{d} x \mathrm{d} t\label{tail3}\\
&+ J_\ve^1+ J_\ve^2,\notag
		\end{align}
	\end{subequations}
	where 
	\begin{align}
	J_\ve^1
	\triangleq&\int_0^T \!\int_{\Omega}   \nabla \HH: \nn_\ve \otimes\nn_{\varepsilon}\(|\nabla \psi_\ve|-\ve |\nabla Q_\ve|^2\)\, \mathrm{d} x \mathrm{d} t\nonumber\\
      &+\ve \int_0^T \!\int_{\Omega}   \nabla \HH:(\nn_\ve\otimes \nn_\ve)\(
	|\nabla Q_\ve|^2-|\Pi_{Q_\ve} \nabla Q_\ve|^2\)\, \mathrm{d} x \mathrm{d} t \nonumber\\
	&-\ve\int_0^T \!\int_{\Omega}   \sum_{i,j=1}^3(\nabla \HH)_{ij}   \Big((\p_i Q_\ve-\Pi_{Q_\ve} \p_i Q_\ve):(\p_j Q_\ve-\Pi_{Q_\ve} \p_j Q_\ve)\Big)\, \mathrm{d} x \mathrm{d} t ,\label{J1}\\
	J_\ve^2\triangleq &-\int_0^T \!\int_{\Omega}   \(\p_t  \bxi+\left((\HH+\vv) \cdot \nabla\right) \bxi 
	+\left(\nabla \HH+\nabla \tilde{\vv}\right)^{T} \bxi\)\cdot (\nn_\ve-\bxi) |\nabla \psi_\ve|\, \mathrm{d} x \mathrm{d} t\nonumber \\
	&-\int_0^T \!\int_{\Omega}   \Big(\p_t  \bxi+\left((\HH+\vv) \cdot \nabla\right) \bxi \Big)\cdot \bxi\, |\nabla \psi_\ve|\, \mathrm{d} x \mathrm{d} t-\int_0^T \!\int_{\Omega}   \left(\nabla \tilde{\vv}\right)^{T} :(\bxi\otimes \bxi) |\nabla \psi_\ve|\,\mathrm{d} x \mathrm{d} t.\label{J2}
	\end{align}
\end{proposition}

Before proving it, we need the following lemma, in which a key identity has been provided.
\begin{lemma}\label{lemma:expansion 1}
Under the construction \eqref{2.23.2}, the following integral identity holds over the domain $\Omega \times [0, T]$:
	\begin{equation}
	\begin{split} \label{expansion1}
	&\int_0^T \!\int_{\Omega} \nabla \tilde{\vv}: (\bxi \otimes\nn_{\varepsilon})\left|\nabla \psi_{\varepsilon}\right| \, \mathrm{d} x \mathrm{d} t
	-\int_0^T \!\int_{\Omega} (\operatorname{div} \tilde{\vv}) \, \bxi  \cdot \nabla \psi_{\varepsilon} \, \mathrm{d} x \mathrm{d} t \\
	=&\int_0^T \!\int_{\Omega} \nabla \tilde{\vv}: (\bxi-\nn_\ve) \otimes\nn_{\varepsilon}\left|\nabla \psi_{\varepsilon}\right|\, \mathrm{d} x \mathrm{d} t+\int_0^T \!\int_{\Omega} \HH_\ve\cdot \tilde{\vv} |\nabla Q_\ve|\, \mathrm{d} x \mathrm{d} t \\
	&+\int_0^T \!\int_{\Omega}\operatorname{div} \tilde{\vv}  \( \frac{\ve}2 |\nabla Q_\ve|^2  +\frac{1}\ve F_\ve(Q_\ve) -|\nabla \psi_\ve| \)\, \mathrm{d} x \mathrm{d} t +\int_0^T \!\int_{\Omega}\operatorname{div} \tilde{\vv}  |\nabla\psi_\ve|(1-\bxi\cdot\nn_{\varepsilon})\, \mathrm{d} x \mathrm{d} t \\
	&- \ve\int_0^T \!\int_{\Omega} \sum_{i,j=1}^3(\nabla \tilde{\vv})_{ij}\(\p_i Q_\ve: \p_j Q_\ve   \)\, \mathrm{d} x \mathrm{d} t +\int_0^T \!\int_{\Omega} \nabla \tilde{\vv}: (\nn_\ve \otimes\nn_{\varepsilon})\left|\nabla \psi_{\varepsilon}\right|\, \mathrm{d} x \mathrm{d} t.
	\end{split}
	\end{equation}
This identity still holds true when $\tilde{\vv} $ is replaced by $\HH$.
\end{lemma}
\begin{proof}
	Since
\begin{align*}
	&\int_0^T \!\int_{\Omega} \nabla \tilde{\vv}: \(\bxi \otimes\nn_{\varepsilon}\)\left|\nabla \psi_{\varepsilon}\right| \, \mathrm{d} x \mathrm{d} t
	-\int_0^T \!\int_{\Omega} (\operatorname{div} \tilde{\vv}) \, \bxi  \cdot \nabla \psi_{\varepsilon} \, \mathrm{d} x \mathrm{d} t \\
	=&\int_0^T \!\int_{\Omega} \nabla \tilde{\vv}: (\bxi-\nn_\ve) \otimes\nn_{\varepsilon}\left|\nabla \psi_{\varepsilon}\right|\, \mathrm{d} x \mathrm{d} t+\int_0^T \!\int_{\Omega} \nabla \tilde{\vv}: \(\nn_\ve \otimes\nn_{\varepsilon}\)\left|\nabla \psi_{\varepsilon}\right|\, \mathrm{d} x \mathrm{d} t \\
&-\int_0^T \!\int_{\Omega} (\operatorname{div} \tilde{\vv})\, \bxi  \cdot \nabla \psi_{\varepsilon} \, \mathrm{d} x \mathrm{d} t,
\end{align*}
	we simplify the problem by calculating the right-hand-side of \eqref{expansion1}.

For this purpose, we define the energy stress tensor $T_\ve$ by
\begin{align} \label{24.1.22.1}
(T_\ve)_{ij}= \( \frac{\ve}2 |\nabla Q_\ve|^2 +\frac{1}{\ve} F_\ve(Q_\ve)  \) \delta_{ij} - \ve \p_i Q_\ve: \p_j Q_\ve.   \end{align}
	From the definition of \eqref{mean curvature app},
	we derive
	\begin{equation}\label{variation1}
	\operatorname{div} T_\ve
	=-\ve \nabla Q_\ve : \Delta Q_\ve
	+ \frac{1}{\ve} D  F_\ve(Q_\ve) : \nabla Q_\ve=\HH_\ve |\nabla Q_\ve|.
	\end{equation}
	Testing this identity by $\tilde{\vv}$ and integrating by parts, we can conclude that
	\begin{equation}
	\begin{split}
	&\int_0^T \!\int_{\Omega} \HH_\ve\cdot \tilde{\vv} |\nabla Q_\ve|\,\mathrm{d} x \mathrm{d} t  =- \int_0^T \!\int_{\Omega} \nabla \tilde{\vv} \colon T_\ve \,\mathrm{d} x \mathrm{d} t,\\
	=&-  \int_0^T \!\int_{\Omega}\operatorname{div} \tilde{\vv}  \( \frac{\ve}2 |\nabla Q_\ve|^2 +\frac{1}{\ve} F_\ve(Q_\ve)  \)\, \mathrm{d} x \mathrm{d} t+ \e\int_0^T \!\int_{\Omega} \sum_{i,j=1}^3(\nabla \tilde{\vv})_{ij} \(\p_i Q_\ve: \p_j Q_\ve\) \mathrm{d} x \mathrm{d} t,
	\end{split}
	\end{equation}
	hence, by adding zero, we rewrite
	\begin{align*}
	&\int_0^T \!\int_{\Omega} \nabla \tilde{\vv}: \nn_\ve \otimes\nn_{\varepsilon}\left|\nabla \psi_{\varepsilon}\right|\mathrm{d} x \mathrm{d} t\\
	=&\int_0^T \!\int_{\Omega} \HH_\ve\cdot \tilde{\vv} |\nabla Q_\ve|\, \mathrm{d} x \mathrm{d} t+\int_0^T \!\int_{\Omega}\operatorname{div} \tilde{\vv}  \( \frac{\ve}2 |\nabla Q_\ve|^2 +\frac{1}{\ve} F_\ve(Q_\ve) -|\nabla \psi_\ve| \)\,\mathrm{d} x \mathrm{d} t\\
     &+\int_0^T \!\int_{\Omega}\operatorname{div} \tilde{\vv} |\nabla\psi_\ve|\,\mathrm{d} x \mathrm{d} t-\e\int_0^T \!\int_{\Omega} \sum_{i,j=1}^3(\nabla \tilde{\vv})_{ij}\(\p_i Q_\ve: \p_j Q_\ve   \) \mathrm{d} x \mathrm{d} t\\
     &+\int_0^T \!\int_{\Omega} (\nabla \tilde{\vv}): (\nn_\ve \otimes\nn_{\varepsilon})\left|\nabla \psi_{\varepsilon}\right|\mathrm{d} x \mathrm{d} t.
	\end{align*}
	After all of these derivations lead us to the desired result in  \eqref{expansion1}.
\end{proof}

\begin{proof}[Proof of Proposition \ref{lemma exact dt relative entropy}]
We obtains from \eqref{energy inequality} that
\begin{eqnarray}
\begin{split} \label{2.18.5}
&\frac{1}{2} \int_{\Omega} |\mathbf{v}_\varepsilon(T)|^2\mathrm{d} x +\int_{\Omega} \left[\frac{\varepsilon}{2}|\nabla Q_{\varepsilon}|^2+ \frac{1}{\varepsilon}F_\e (Q_{\varepsilon})\right](T)\, \mathrm{d} x\\
&+\int_0^T \!\int_{\Omega} |\nabla \mathbf{v}_\varepsilon|^2 \mathrm{d} x \, \mathrm{d} t+\int_0^T \!\int_{\Omega} \frac{1}{\varepsilon} \Big(\varepsilon \Delta Q_{\varepsilon}-\frac{1}{\varepsilon} D F_\e (Q_{\varepsilon})\Big)^{2} \mathrm{d} x \, \mathrm{d} t\\
\le&\frac{1}{2} \int_{\Omega} |\mathbf{v}_\varepsilon(0)|^2\mathrm{d} x+\int_{\Omega} \left[\frac{\varepsilon}{2}|\nabla Q_{\varepsilon}|^2+ \frac{1}{\varepsilon}F_\e (Q_{\varepsilon})\right](0)\, \mathrm{d} x.
\end{split}
\end{eqnarray}
Thus, combining the above equality with \eqref{8.11.7} and adding zero, we conclude that
\begin{align}
&\int_{\Omega} \frac{1}{2} |\mathbf{v}_\varepsilon(t)|^2+ \left[\frac{\varepsilon}{2}|\nabla Q_{\varepsilon}|^2+ \frac{1}{\varepsilon}F_\e (Q_{\varepsilon})-\boldsymbol{\xi} \cdot \nabla \psi_{\varepsilon}\right](t)\mathrm{d} x \bigg|_{t=0}^{t=T} \notag\\
&+\int_0^T \!\int_{\Omega} |\nabla \mathbf{v}_\varepsilon|^2 \mathrm{d} x \, \mathrm{d} t+\int_0^T \!\int_{\Omega} \frac{1}{\varepsilon}\Big(\varepsilon \Delta Q_{\varepsilon}-\frac{1}{\varepsilon} D F_\e (Q_{\varepsilon})\Big)^{2} \mathrm{d} x \mathrm{d} t \notag\\
\le&\int_0^T \!\int_{\Omega} ( \operatorname{div} \boldsymbol{\xi}) \,D d^F_\e(Q_\e) \colon \p_t Q_\e\, \mathrm{d} x \mathrm{d} t  \notag\\
&+\int_0^T \!\int_{\Omega}  \left((\HH+\vv) \cdot \nabla\right) \bxi\cdot\nabla \psi_\ve\,\mathrm{d} x \mathrm{d} t +\int_0^T \!\int_{\Omega}  \left(\nabla \HH+\nabla \tilde{\vv}\right)^{T} \bxi \cdot\nabla \psi_\ve\,\mathrm{d} x \mathrm{d} t\notag\\
& -\int_0^T \!\int_{\Omega}  \(\p_t  \bxi+\left((\HH+\vv) \cdot \nabla\right) \bxi +\left(\nabla \HH+\nabla \tilde{\vv}\right)^{T} \bxi\)\cdot\nabla \psi_\ve\,\mathrm{d} x \mathrm{d} t \notag.
\end{align}
Since
\begin{align*}
&\int_0^T \!\int_{\Omega} (\HH \cdot \nabla) \boldsymbol{\xi} \cdot \nabla \psi_{\varepsilon} \,\mathrm{d} x \mathrm{d} t =-\int_0^T \!\int_{\Omega} \HH \otimes \boldsymbol{\xi} : \nabla^2 \psi_{\varepsilon} \,\mathrm{d} x \mathrm{d} t-\int_0^T \!\int_{\Omega} (\operatorname{div}\HH)\, \boldsymbol{\xi}\cdot \nabla \psi_{\varepsilon} \,\mathrm{d} x \mathrm{d} t\\
=&\int_0^T \!\int_{\Omega} (\boldsymbol{\xi}\cdot \nabla) \HH \cdot \nabla \psi_{\varepsilon} \,\mathrm{d} x \mathrm{d} t+\int_0^T \!\int_{\Omega} (\operatorname{div}\boldsymbol{\xi}) \,\HH \cdot \nabla \psi_{\varepsilon} \,\mathrm{d} x \mathrm{d} t-\int_0^T \!\int_{\Omega} (\operatorname{div}\HH)\, \boldsymbol{\xi}\cdot \nabla \psi_{\varepsilon} \,\mathrm{d} x \mathrm{d} t,
\end{align*}
we have
\begin{align}
&\int_{\Omega} \frac{1}{2} |\mathbf{v}_\varepsilon(t)|^2+ \left[\frac{\varepsilon}{2}|\nabla Q_{\varepsilon}|^2+ \frac{1}{\varepsilon}F_\e (Q_{\varepsilon})-\boldsymbol{\xi} \cdot \nabla \psi_{\varepsilon}\right](t)\mathrm{d} x \bigg|_{t=0}^{t=T} \notag\\
&+\int_0^T \!\int_{\Omega} |\nabla \mathbf{v}_\varepsilon|^2 \mathrm{d} x \, \mathrm{d} t+\int_0^T \!\int_{\Omega} \frac{1}{\varepsilon}\Big(\varepsilon \Delta Q_{\varepsilon}-\frac{1}{\varepsilon} D F_\e (Q_{\varepsilon})\Big)^{2} \mathrm{d} x \mathrm{d} t \notag\\
\le&\int_0^T \!\int_{\Omega}  (\operatorname{div} \boldsymbol{\xi}) \,D d^F_\e(Q_\e) \colon \p_t Q_\e\, \mathrm{d} x \mathrm{d} t  \notag\\
&+\int_0^T \!\int_{\Omega}  \left(\vv \cdot \nabla\right) \bxi\cdot\nabla \psi_\ve\,\mathrm{d} x \mathrm{d} t +\int_0^T \!\int_{\Omega} (\boldsymbol{\xi}\cdot \nabla) \HH \cdot \nabla \psi_{\varepsilon} \,\mathrm{d} x \mathrm{d} t+\int_0^T \!\int_{\Omega} (\operatorname{div}\boldsymbol{\xi}) \,\HH \cdot \nabla \psi_{\varepsilon} \,\mathrm{d} x \mathrm{d} t\notag\\
&-\int_0^T \!\int_{\Omega} (\operatorname{div}\HH)\, \boldsymbol{\xi}\cdot \nabla \psi_{\varepsilon} \,\mathrm{d} x \mathrm{d} t+\int_0^T \!\int_{\Omega} (\nn_\e \cdot \nabla ) \HH \cdot \bxi |\nabla \psi_\ve| \,\mathrm{d} x \mathrm{d} t\label{24.1.12.16.23} \\
& +\int_0^T \!\int_{\Omega} (\nn_\e \cdot \nabla ) \tilde{\vv} \cdot \bxi |\nabla \psi_\ve| \,\mathrm{d} x \mathrm{d} t -\int_0^T \!\int_{\Omega}  \(\p_t  \bxi+\left((\HH+\vv) \cdot \nabla\right) \bxi +\left(\nabla \HH+\nabla \tilde{\vv}\right)^{T} \bxi\)\cdot\nabla \psi_\ve\,\mathrm{d} x \mathrm{d} t\notag.
\end{align}
Replacing $\tilde{\vv}$ with $\mathbf{H}$ in Lemma \ref{lemma:expansion 1}, and then applying the resulting lemma to \eqref{24.1.12.16.23}, combined with $(\nn_\e \cdot \nabla) \tilde{\vv} \cdot \bxi=\nabla \tilde{\vv}: \bxi \otimes \nn_\e $, we obtain:
\begin{align*}
&\int_{\Omega} \frac{1}{2} |\mathbf{v}_\varepsilon(t)|^2+ \left[\frac{\varepsilon}{2}|\nabla Q_{\varepsilon}|^2+ \frac{1}{\varepsilon}F_\e (Q_{\varepsilon})-\boldsymbol{\xi} \cdot \nabla \psi_{\varepsilon}\right](t)\mathrm{d} x \bigg|_{t=0}^{t=T} \notag\\
&+\int_0^T \!\int_{\Omega} |\nabla \mathbf{v}_\varepsilon|^2 \mathrm{d} x \, \mathrm{d} t+\int_0^T \!\int_{\Omega} \frac{1}{\varepsilon}\Big(\varepsilon \Delta Q_{\varepsilon}-\frac{1}{\varepsilon} D F_\e (Q_{\varepsilon})\Big)^{2} \mathrm{d} x \mathrm{d} t \notag\\
\le&\int_0^T \!\int_{\Omega}  (\operatorname{div} \boldsymbol{\xi}) \,D d^F_\e(Q_\e) \colon \p_t Q_\e\, \mathrm{d} x \mathrm{d} t +\int_0^T \!\int_{\Omega} \HH_\ve\cdot \HH |\nabla Q_\ve|\, \mathrm{d} x \mathrm{d} t \notag\\
&+\int_0^T \!\int_{\Omega}  \left(\vv \cdot \nabla\right) \bxi\cdot\nabla \psi_\ve\,\mathrm{d} x \mathrm{d} t +\int_0^T \!\int_{\Omega} (\operatorname{div}\boldsymbol{\xi}) \,\HH \cdot \nabla \psi_{\varepsilon} \,\mathrm{d} x \mathrm{d} t +\int_0^T \!\int_{\Omega} \nabla \tilde{\vv}: (\bxi \otimes \nn_{\varepsilon})\left|\nabla \psi_{\varepsilon}\right| \,\mathrm{d} x \mathrm{d} t \notag\\
&+\int_0^T \!\int_{\Omega} (\boldsymbol{\xi}\cdot \nabla) \HH \cdot \nabla \psi_{\varepsilon} \,\mathrm{d} x \mathrm{d} t+\int_0^T \!\int_{\Omega} \nabla \HH: (\bxi-\nn_\ve) \otimes\nn_{\varepsilon}\left|\nabla \psi_{\varepsilon}\right|\, \mathrm{d} x \mathrm{d} t\\
&+\int_0^T \!\int_{\Omega}\operatorname{div} \HH \( \frac{\ve}2 |\nabla Q_\ve|^2  +\frac{1}\ve F_\ve(Q_\ve) -|\nabla \psi_\ve| \)\, \mathrm{d} x \mathrm{d} t +\int_0^T \!\int_{\Omega}\operatorname{div} \HH |\nabla\psi_\ve|(1-\bxi\cdot\nn_{\varepsilon})\, \mathrm{d} x \mathrm{d} t\nonumber\\
&-\e\int_0^T \!\int_{\Omega} \sum_{i,j=1}^3(\nabla \HH)_{ij}\(\p_i Q_\ve: \p_j Q_\ve   \)\, \mathrm{d} x \mathrm{d} t +\int_0^T \!\int_{\Omega} \nabla \HH: (\nn_\ve \otimes\nn_{\varepsilon})\left|\nabla \psi_{\varepsilon}\right|\, \mathrm{d} x \mathrm{d} t\\
&-\int_0^T \!\int_{\Omega}  \(\p_t  \bxi+\left((\HH+\vv) \cdot \nabla\right) \bxi +\left(\nabla \HH+\nabla \tilde{\vv}\right)^{T} \bxi\)\cdot (\nn_{\varepsilon}-\bxi) |\nabla \psi_\ve|\,\mathrm{d} x \mathrm{d} t\\
&-\int_0^T \!\int_{\Omega}  \Big(\p_t  \bxi+\left((\HH+\vv) \cdot \nabla\right) \bxi \Big)\cdot \bxi |\nabla \psi_\ve|\,\mathrm{d} x \mathrm{d} t-\int_0^T \!\int_{\Omega}  \left(\nabla \HH+\nabla \tilde{\vv}\right)^{T}: \bxi \otimes \bxi |\nabla \psi_\ve|\,\mathrm{d} x \mathrm{d} t.
\end{align*}
Note that in the fifth line, $\int_0^T \!\int_{\Omega} (\boldsymbol{\xi}\cdot \nabla) \HH \cdot \nabla \psi_{\varepsilon} \,\mathrm{d} x \mathrm{d} t$ and $\int_0^T \!\int_{\Omega} \nabla \HH: (\bxi-\nn_\ve) \otimes\nn_{\varepsilon}\left|\nabla \psi_{\varepsilon}\right|\, \mathrm{d} x \mathrm{d} t$ are considered, as well as $\int_0^T \!\int_{\Omega}  \nabla \HH^{T}: \bxi \otimes \bxi |\nabla \psi_\ve|\,\mathrm{d} x \mathrm{d} t$ in the last line. It is known that:
\begin{small}
\begin{align*}
&\int_0^T \!\int_{\Omega} (\boldsymbol{\xi}\cdot \nabla) \HH \cdot \nabla \psi_{\varepsilon} \,\mathrm{d} x \mathrm{d} t + \int_0^T \!\int_{\Omega} \nabla \HH: (\bxi-\nn_\ve) \otimes\nn_{\varepsilon}\left|\nabla \psi_{\varepsilon}\right|\, \mathrm{d} x \mathrm{d} t - \int_0^T \!\int_{\Omega}  \nabla \HH^{T}: \bxi \otimes \bxi |\nabla \psi_\ve|\,\mathrm{d} x \mathrm{d} t\\
=&\int_0^T \!\int_{\Omega} \nabla \HH: (\nn_\ve-\bxi) \otimes \bxi \left|\nabla \psi_{\varepsilon}\right|\, \mathrm{d} x \mathrm{d} t + \int_0^T \!\int_{\Omega} \nabla \HH: (\bxi-\nn_\ve) \otimes\nn_{\varepsilon}\left|\nabla \psi_{\varepsilon}\right|\, \mathrm{d} x \mathrm{d} t \\
=& -\int_0^T \!\int_{\Omega} \nabla \HH: (\bxi-\nn_\ve) \otimes (\bxi-\nn_\ve) \left|\nabla \psi_{\varepsilon}\right|\, \mathrm{d} x \mathrm{d} t.
\end{align*}
\end{small}
And the remaining terms in the last two lines are defined as $J_\ve^2$. In fact, based on the orthogonality of the orthogonal projection \eqref{projection1}, we derive that 
\begin{align} \label{24.1.20.0}
(\p_i Q_\ve-\Pi_{Q_\ve} \p_i Q_\ve):\Pi_{Q_\ve} \p_j Q_\ve =\Pi_{Q_\ve} \p_i Q_\ve:(\p_j Q_\ve-\Pi_{Q_\ve} \p_j Q_\ve)=0.
\end{align} 
Additionally, by applying \eqref{projection}, we can deduce that the third to last line is equivalent to the presence of $J_\ve^1$ on the right-hand side of \eqref{time deri 4}:
\begin{small}
\begin{eqnarray}
\begin{split} \label{hungry}
&-\e\int_0^T \!\int_{\Omega} \sum_{i,j=1}^3(\nabla \HH)_{ij}\(\p_i Q_\ve: \p_j Q_\ve   \)\, \mathrm{d} x \mathrm{d} t +\int_0^T \!\int_{\Omega} \nabla \HH: (\nn_\ve \otimes\nn_{\varepsilon})\left|\nabla \psi_{\varepsilon}\right|\, \mathrm{d} x \mathrm{d} t\\
=&\int_0^T \!\int_{\Omega} \nabla \HH: \nn_\ve \otimes\nn_{\varepsilon}\left|\nabla \psi_{\varepsilon}\right|\,\mathrm{d} x \mathrm{d} t-\ve\int_0^T \!\int_{\Omega} \sum_{i,j=1}^3(\nabla \HH)_{ij}(\Pi_{Q_\ve} \p_i Q_\ve :\Pi_{Q_\ve} \p_j Q_\ve)\,\mathrm{d} x \mathrm{d} t\\
&-\e\int_0^T \!\int_{\Omega} \sum_{i,j=1}^3(\nabla \HH)_{ij} \, \Big((\p_i Q_\ve-\Pi_{Q_\ve} \p_i Q_\ve):(\p_j Q_\ve-\Pi_{Q_\ve} \p_j Q_\ve)\Big) \, \mathrm{d} x \mathrm{d} t\\
=&\int_0^T \!\int_{\Omega} \nabla \HH: \nn_\ve \otimes\nn_{\varepsilon}\(|\nabla \psi_\ve|-\ve |\nabla Q_\ve|^2\)\,\mathrm{d} x \mathrm{d} t+\ve \int_0^T \!\int_{\Omega} \nabla \HH:(\nn_\ve\otimes \nn_\ve)\( |\nabla Q_\ve|^2-|\Pi_{Q_\ve} \nabla Q_\ve|^2\)\, \mathrm{d} x \mathrm{d} t \\
&- \e\int_0^T \!\int_{\Omega} \sum_{i,j=1}^3(\nabla \HH)_{ij} \, \Big((\p_i Q_\ve-\Pi_{Q_\ve} \p_i Q_\ve):(\p_j Q_\ve-\Pi_{Q_\ve} \p_j Q_\ve)\Big) \, \mathrm{d} x \mathrm{d} t \triangleq J_\ve^1.
\end{split}
\end{eqnarray}
\end{small}
Therefore,
\begin{align*}
&\int_{\Omega} \frac{1}{2} |\mathbf{v}_\varepsilon(t)|^2+ \left[\frac{\varepsilon}{2}|\nabla Q_{\varepsilon}|^2+ \frac{1}{\varepsilon}F_\e (Q_{\varepsilon})-\boldsymbol{\xi} \cdot \nabla \psi_{\varepsilon}\right](t)\mathrm{d} x \bigg|_{t=0}^{t=T} \notag\\
&+\int_0^T \!\int_{\Omega} |\nabla \mathbf{v}_\varepsilon|^2 \mathrm{d} x \, \mathrm{d} t+\int_0^T \!\int_{\Omega} \frac{1}{\varepsilon}\Big(\varepsilon \Delta Q_{\varepsilon}-\frac{1}{\varepsilon} D F_\e (Q_{\varepsilon})\Big)^{2} \mathrm{d} x \mathrm{d} t \notag\\
\le& \int_0^T \!\int_{\Omega}  (\operatorname{div} \boldsymbol{\xi}) \,D d^F_\e(Q_\e) \colon \p_t Q_\e\, \mathrm{d} x \mathrm{d} t  +\int_0^T \!\int_{\Omega} \HH_\ve\cdot \HH |\nabla Q_\ve|\, \mathrm{d} x \mathrm{d} t  \notag\\
&+\int_0^T \!\int_{\Omega}  \left(\vv \cdot \nabla\right) \bxi\cdot\nabla \psi_\ve\,\mathrm{d} x \mathrm{d} t +\int_0^T \!\int_{\Omega} (\operatorname{div}\boldsymbol{\xi}) \,\HH \cdot \nabla \psi_{\varepsilon} \,\mathrm{d} x \mathrm{d} t +\int_0^T \!\int_{\Omega} \nabla \tilde{\vv}: (\bxi \otimes \nn_{\varepsilon})\left|\nabla \psi_{\varepsilon}\right| \,\mathrm{d} x \mathrm{d} t \notag\\
&-\int_0^T \!\int_{\Omega} \nabla \HH: (\bxi-\nn_\ve) \otimes (\bxi-\nn_\ve) \left|\nabla \psi_{\varepsilon}\right|\, \mathrm{d} x \mathrm{d} t\nonumber\\
&+\int_0^T \!\int_{\Omega}\operatorname{div} \HH \( \frac{\ve}2 |\nabla Q_\ve|^2  +\frac{1}\ve F_\ve(Q_\ve) -|\nabla \psi_\ve| \)\, \mathrm{d} x \mathrm{d} t +\int_0^T \!\int_{\Omega}\operatorname{div} \HH |\nabla\psi_\ve|(1-\bxi\cdot\nn_{\varepsilon})\, \mathrm{d} x \mathrm{d} t\nonumber\\
&+J_\ve^1+J_\ve^2.
\end{align*}
Recall \eqref{10.19.1}, which gives $\partial_{t} Q_{\varepsilon}+(\mathbf{v}_{\varepsilon} \cdot \nabla) Q_{\varepsilon} =\Delta Q_{\varepsilon}-\frac{1}{\varepsilon^{2}} D F (Q_{\varepsilon})$. Hence we obtain
\begin{small}
\begin{align*}
&-\int_0^T \!\int_{\Omega} \frac{1}{\varepsilon}\Big(\varepsilon \Delta Q_{\varepsilon}-\frac{1}{\varepsilon} D F_\e (Q_{\varepsilon})\Big)^{2} \mathrm{d} x \mathrm{d} t +\int_0^T \!\int_{\Omega}  (\operatorname{div} \boldsymbol{\xi}) \,D d^F_\e(Q_\e) \colon \p_t Q_\e\, \mathrm{d} x \mathrm{d} t  \notag\\
&+\int_0^T \!\int_{\Omega} \HH_\ve\cdot \HH |\nabla Q_\ve|\, \mathrm{d} x \mathrm{d} t +\int_0^T \!\int_{\Omega}(\operatorname{div}\boldsymbol{\xi}) \,\HH \cdot \nabla \psi_{\varepsilon} \,\mathrm{d} x \mathrm{d} t  \notag\\
=& -\int_0^T \!\int_{\Omega} {\varepsilon}\Big(\partial_{t} Q_{\varepsilon}+(\mathbf{v}_{\varepsilon} \cdot \nabla) Q_{\varepsilon}\Big)^{2} \mathrm{d} x \mathrm{d} t +\int_0^T \!\int_{\Omega}  (\operatorname{div} \boldsymbol{\xi}) \,D d^F_\e(Q_\e) \colon \Big(\partial_{t} Q_{\varepsilon}+(\mathbf{v}_{\varepsilon} \cdot \nabla) Q_{\varepsilon}\Big)\, \mathrm{d} x \mathrm{d} t  \notag\\
&-\int_0^T \!\int_{\Omega} (\operatorname{div} \boldsymbol{\xi}) \,D d^F_\e(Q_\e) \colon (\mathbf{v}_{\varepsilon} \cdot \nabla) Q_{\varepsilon}\, \mathrm{d} x \mathrm{d} t +\int_0^T \!\int_{\Omega} (\operatorname{div}\boldsymbol{\xi}) \,\HH \cdot \nabla \psi_{\varepsilon} \,\mathrm{d} x \mathrm{d} t +\int_0^T \!\int_{\Omega} \HH_\ve\cdot \HH |\nabla Q_\ve|\, \mathrm{d} x \mathrm{d} t  \notag\\
=& -\frac1{2\ve} \int_0^T \!\int_{\Omega}  \Big|\ve \Big(\partial_{t} Q_{\varepsilon}+(\mathbf{v}_{\varepsilon} \cdot \nabla) Q_{\varepsilon}\Big)\Big|^2\\
&\qquad\qquad -2(\div \bxi)  D d^F_\ve(Q_\ve): \ve \Big(\partial_{t} Q_{\varepsilon}+(\mathbf{v}_{\varepsilon} \cdot \nabla) Q_{\varepsilon}\Big)+(\div \bxi)^2 | D d^F_\ve(Q_\ve)|^2 \, \mathrm{d} x \mathrm{d} t\notag\\
& - \int_0^T \!\int_{\Omega} \frac1{2\ve} \Big|\ve \Big(\partial_{t} Q_{\varepsilon}+(\mathbf{v}_{\varepsilon} \cdot \nabla) Q_{\varepsilon}\Big)\Big|^2 \, \mathrm{d} x \mathrm{d} t+ \int_0^T \!\int_{\Omega} \frac1{2\ve}(\div \bxi)^2 | D d^F_\ve(Q_\ve)|^2 \, \mathrm{d} x \mathrm{d} t \notag\\
&-\int_0^T \!\int_{\Omega}  (\operatorname{div} \boldsymbol{\xi}) \,D d^F_\e(Q_\e) \colon (\mathbf{v}_{\varepsilon} \cdot \nabla) Q_{\varepsilon}\, \mathrm{d} x \mathrm{d} t +\int_0^T \!\int_{\Omega} (\operatorname{div}\boldsymbol{\xi}) \,\HH \cdot \nabla \psi_{\varepsilon} \,\mathrm{d} x \mathrm{d} t \notag\\
& - \frac1{2\ve}\int_0^T \!\int_{\Omega}  |\HH_\ve|^2 - 2\HH_\ve\cdot \ve |\nabla Q_\ve| \HH + \ve^2 |\nabla Q_\ve|^2 |\HH|^2 \,\mathrm{d} x \mathrm{d} t
	+ \frac1{2\ve} \int_0^T \!\int_{\Omega}  |\HH_\ve|^2 + \ve^2 |\nabla Q_\ve|^2 |\HH|^2 \,\mathrm{d} x \mathrm{d} t \notag\\
=&  -\frac1{2\ve} \int_0^T \!\int_{\Omega} \Big|\ve \Big(\partial_{t} Q_{\varepsilon}+(\mathbf{v}_{\varepsilon} \cdot \nabla) Q_{\varepsilon}\Big)- (\div \bxi) D d^F_\ve(Q_\ve) \Big|^2
 \,\mathrm{d} x \mathrm{d} t - \frac1{2\ve} \int_0^T \!\int_{\Omega} \Big|\HH_\ve - \ve |\nabla Q_\ve| \HH \Big|^2 \,\mathrm{d} x \mathrm{d} t \notag\\
& - \frac1{2\ve} \int_0^T \!\int_{\Omega}  \Big|\ve \Big(\partial_{t} Q_{\varepsilon}+(\mathbf{v}_{\varepsilon} \cdot \nabla) Q_{\varepsilon}\Big)\Big|^2-|\HH_\ve|^2  \,\mathrm{d} x \mathrm{d} t -\int_0^T \!\int_{\Omega}  (\operatorname{div} \boldsymbol{\xi}) \,D d^F_\e(Q_\e) \colon (\mathbf{v}_{\varepsilon} \cdot \nabla) Q_{\varepsilon}\, \mathrm{d} x \mathrm{d} t \notag\\
& + \frac1{2\ve} \int_0^T \!\int_{\Omega}  (\div \bxi)^2  |D d^F_\ve(Q_\ve)|^2 + 2\ve(\div \bxi)  \HH \cdot \nabla \psi_\ve + |\ve \Pi_{Q_\ve}\nabla Q_\ve|^2 |\HH|^2 \,\mathrm{d} x \mathrm{d} t \\
&+\frac\ve{2} \int_0^T \!\int_{\Omega} \left(|\nabla Q_\ve|^2- | \Pi_{Q_\ve}\nabla Q_\ve|^2\right)  |\HH|^2 \,\mathrm{d} x \mathrm{d} t \notag \\
=&  -\frac1{2\ve} \int_0^T \!\int_{\Omega} \Big|\ve \Big(\partial_{t} Q_{\varepsilon}+(\mathbf{v}_{\varepsilon} \cdot \nabla) Q_{\varepsilon}\Big)- (\div \bxi) D d^F_\ve(Q_\ve) \Big|^2
 \,\mathrm{d} x \mathrm{d} t - \frac1{2\ve} \int_0^T \!\int_{\Omega} \Big|\HH_\ve - \ve |\nabla Q_\ve| \HH \Big|^2 \,\mathrm{d} x \mathrm{d} t \notag\\
& - \frac1{2\ve} \int_0^T \!\int_{\Omega}  \Big|\ve \Big(\partial_{t} Q_{\varepsilon}+(\mathbf{v}_{\varepsilon} \cdot \nabla) Q_{\varepsilon}\Big)\Big|^2-|\HH_\ve|^2  \,\mathrm{d} x \mathrm{d} t -\int_0^T \!\int_{\Omega}  (\operatorname{div} \boldsymbol{\xi}) \,D d^F_\e(Q_\e) \colon (\mathbf{v}_{\varepsilon} \cdot \nabla) Q_{\varepsilon}\, \mathrm{d} x \mathrm{d} t \notag\\
& +\frac 1{2\ve} \int_0^T \!\int_{\Omega}   \Big| (\operatorname{div} \bxi) |D d^F_\ve(Q_\ve)|\nn_\ve +\ve |\Pi_{Q_\ve} \nabla Q_\ve| \HH\Big|^2\,\mathrm{d} x \mathrm{d} t +\frac\ve{2} \int_0^T \!\int_{\Omega} \left(|\nabla Q_\ve|^2- | \Pi_{Q_\ve}\nabla Q_\ve|^2\right)  |\HH|^2 \,\mathrm{d} x \mathrm{d} t ,
\end{align*}
\end{small}
where \eqref{projectionnorm} is used in the last step. Combining all of these calculations, we finish the proof.
\end{proof}

By virtue of Proposition \ref{lemma exact dt relative entropy}, it is direct to establish the following estimates.
\begin{proposition} \label{lemma 3.4.26}
Let $(\mathbf{v}_{\varepsilon},Q_{\varepsilon})$ be the weak solution as described in Definition \ref{weak solutions}, then there exists a positive constant $C$, such that the following inequality holds:
\begin{subequations} \label{hope}  
		\begin{align}
&\int_{\Omega} \frac{1}{2} |\mathbf{v}_\varepsilon(t)|^2+ \left[\frac{\varepsilon}{2}|\nabla Q_{\varepsilon}|^2+ \frac{1}{\varepsilon}F_\e (Q_{\varepsilon})-\boldsymbol{\xi} \cdot \nabla \psi_{\varepsilon}\right](t)\mathrm{d} x \bigg|_{t=0}^{t=T} \notag\\
&+\int_0^T \!\int_{\Omega} |\nabla \mathbf{v}_\varepsilon|^2 \mathrm{d} x \, \mathrm{d} t+\frac 1{2\ve}\int_0^T \!\int_{\Omega}  \left| \ve \Big(\partial_{t} Q_{\varepsilon}+(\mathbf{v}_{\varepsilon} \cdot \nabla) Q_{\varepsilon}\Big)  -(\operatorname{div} \bxi) D d^F_\ve(Q_\ve)  \right|^2\mathrm{d} x \mathrm{d} t\notag\\
		&+\frac 1{2\ve}\int_0^T \!\int_{\Omega}  \Big| \HH_\ve-\ve|\nabla Q_\ve| \HH \Big|^2 \, \mathrm{d} x \mathrm{d} t+{\frac 1{2\ve}\int_0^T \!\int_{\Omega}    \(\Big|\ve \Big(\partial_{t} Q_{\varepsilon}+(\mathbf{v}_{\varepsilon} \cdot \nabla) Q_{\varepsilon}\Big)\Big|^2 -   |\HH_\ve|^2\)\,\mathrm{d} x \mathrm{d} t} \label{8.14.21} \\
		\le& C\int_0^T E \left[\mathbf{v}_{\varepsilon}, Q_{\varepsilon} \mid \mathbf{v}, \chi\right]\, \mathrm{d} t-\int_0^T \!\int_{\Omega} (\operatorname{div} \boldsymbol{\xi}) \, D d^F_\e(Q_\e) \colon (\vv_\e \cdot \nabla) Q_\e\, \mathrm{d} x \mathrm{d} t \notag\\
&+ \int_0^T \!\int_{\Omega}   (\vv \cdot \nabla) \bxi\cdot\nabla \psi_\ve+ \nabla \tilde{\vv} : \bxi \otimes \nn_\ve |\nabla \psi_\ve|\,\mathrm{d} x \mathrm{d} t. \notag 
		\end{align}
\end{subequations}
\end{proposition}
\begin{proof}
The estimates are analogous to \cite[Proposition 4.3]{liu}. Recalling that
\begin{align} \label{10.8.23}
|\nn_\ve- \bxi |^2 \leq 2 (1-\nn_\ve\cdot\bxi),
\end{align}
then it follows from \eqref{2.20.1} and \eqref{energy bound1} that
\begin{align*}
&\frac 1{2\ve} \int_0^T \!\int_{\Omega}   \Big| (\operatorname{div} \bxi) |D d^F_\ve(Q_\ve)|\nn_\ve +\ve |\Pi_{Q_\ve} \nabla Q_\ve| \HH\Big|^2\,\mathrm{d} x \mathrm{d} t\\
	\leq &\int_0^T \!\int_{\Omega}  \left|(\operatorname{div} \bxi)
	\left(\frac{1}{\sqrt{\ve}} |D d^F_\ve(Q_\ve)|-
	 \sqrt{\ve} |\Pi_{Q_\ve} \nabla Q_\ve| \right)
	\nn_\ve\right|^2
	\,\mathrm{d} x \mathrm{d} t\\
&+\int_0^T \!\int_{\Omega}  \Big|(\operatorname{div}\bxi)
	\sqrt{\ve}  |\Pi_{Q_\ve} \nabla Q_\ve|
	(\nn_\ve-\bxi)\Big|^2
	\,\mathrm{d} x \mathrm{d} t + \int_0^T \!\int_{\Omega} \Big|
	(\HH +(\operatorname{div} \bxi) \bxi )  \sqrt{\ve} |\Pi_{Q_\ve} \nabla Q_\ve| \Big|^2
	\,\mathrm{d} x \mathrm{d} t\\
\le& C \int_0^T E \left[\mathbf{v}_{\varepsilon}, Q_{\varepsilon} \mid \mathbf{v}, \chi\right]\, \mathrm{d} t.
\end{align*}

Using \eqref{energy bound2}, \eqref{energy bound1} and \eqref{10.8.23} again,  it becomes evident that
\begin{align*}
\text{\eqref{tail2}+\eqref{tail3}} \le C \int_0^T E \left[\mathbf{v}_{\varepsilon}, Q_{\varepsilon} \mid \mathbf{v}, \chi\right]\, \mathrm{d} t.
\end{align*}

To control $J_\ve^1$, we use the facts that $\mathbf{n}_{\varepsilon}=\mathbf{n}_{\varepsilon}-\boldsymbol{\xi}+\boldsymbol{\xi}$ and $\nabla \HH:
\mathbf{n}_{\varepsilon} \otimes \boldsymbol{\xi}=(\boldsymbol{\xi} \cdot \nabla) \HH \cdot \mathbf{n}_{\varepsilon}$. By referring to \eqref{2.23.1}, we can see that
\begin{align*}
\int_0^T \!\int_{\Gamma_t(\frac{\delta}{2})} \nabla \HH: \nn_\ve \otimes \bxi \(|\nabla \psi_\ve|-\ve |\nabla Q_\ve|^2\)\, \mathrm{d} x \mathrm{d} t=0.
\end{align*}
Recall \eqref{24.1.20.0} that $(\p_i Q_\ve-\Pi_{Q_\ve} \p_i Q_\ve):\Pi_{Q_\ve} \p_j Q_\ve =\Pi_{Q_\ve} \p_i Q_\ve:(\p_j Q_\ve-\Pi_{Q_\ve} \p_j Q_\ve)=0$, then we can conclude that
$$
|\nabla Q_\ve|^2-|\Pi_{Q_\ve} \nabla Q_\ve|^2 \le \left|\nabla Q_\ve-\Pi_{Q_\ve}\nabla Q_\ve\right|^2.
$$
Consequently,
\begin{align*}
J_\ve^1 =&\int_0^T \!\int_{\Omega}   \nabla \HH: \nn_\ve \otimes\nn_{\varepsilon}\(|\nabla \psi_\ve|-\ve |\nabla Q_\ve|^2\)\, \mathrm{d} x \mathrm{d} t\nonumber\\
      &+\ve \int_0^T \!\int_{\Omega}   \nabla \HH:(\nn_\ve\otimes \nn_\ve)\(
	|\nabla Q_\ve|^2-|\Pi_{Q_\ve} \nabla Q_\ve|^2\)\, \mathrm{d} x \mathrm{d} t \nonumber\\
	&-\ve\int_0^T \!\int_{\Omega}   \sum_{i,j=1}^3(\nabla \HH)_{ij}   \Big((\p_i Q_\ve-\Pi_{Q_\ve} \p_i Q_\ve):(\p_j Q_\ve-\Pi_{Q_\ve} \p_j Q_\ve)\Big)\, \mathrm{d} x \mathrm{d} t \\
\leq &\int_0^T \!\int_{\Omega} \nabla \HH: \big( \nn_\ve \otimes (\nn_\ve-\bxi)\big) \(|\nabla \psi_\ve|-\ve |\nabla Q_\ve|^2\)\, \mathrm{d} x \mathrm{d} t \nonumber\\
	&+\int_0^T \!\int_{\Omega} \nabla \HH: \nn_\ve \otimes \bxi \(|\nabla \psi_\ve|-\ve |\nabla Q_\ve|^2\)\, \mathrm{d} x \mathrm{d} t+C \int_0^T E \left[\mathbf{v}_{\varepsilon}, Q_{\varepsilon} \mid \mathbf{v}, \chi\right]\, \mathrm{d} t\\
	\leq &\|\nabla\HH\|_{L^\infty} \int_0^T \!\int_{\Omega} |\nn_\ve-\bxi| \(\ve |\nabla Q_\ve|^2-\ve |\Pi_{Q_\ve}\nabla Q_\ve|^2+\Big|\ve |\Pi_{Q_\ve}\nabla Q_\ve|^2-|\nabla \psi_\ve|\Big|\)\, \mathrm{d} x \mathrm{d} t\nonumber\\
	&+C \int_0^T \!\int_{\Omega}\min \(d^2_\Gamma,1\)  \(|\nabla \psi_\ve|+\ve |\nabla Q_\ve|^2\)\, \mathrm{d} x \mathrm{d} t +C \int_0^T E \left[\mathbf{v}_{\varepsilon}, Q_{\varepsilon} \mid \mathbf{v}, \chi\right]\, \mathrm{d} t\\
\le& C \int_0^T \!\int_{\Omega} |\nn_\ve-\bxi|  \sqrt{\ve} |\Pi_{Q_\ve}\nabla Q_\ve| \left| \sqrt{\ve} |\Pi_{Q_\ve}\nabla Q_\ve|-\frac{|D d^F_\ve(Q_\ve)|}{\sqrt{\ve}}\right|\, \mathrm{d} x \mathrm{d} t+ C \int_0^T E \left[\mathbf{v}_{\varepsilon}, Q_{\varepsilon} \mid \mathbf{v}, \chi\right]\, \mathrm{d} t\\
\le& C \int_0^T E \left[\mathbf{v}_{\varepsilon}, Q_{\varepsilon} \mid \mathbf{v}, \chi\right]\, \mathrm{d} t,
\end{align*}
where the second-to-last line involves \eqref{projectionnorm}. Additionally, we make frequent use of  \eqref{energy bound2},  \eqref{energy bound1} and \eqref{energy bound3}.

Now we turn our attention to the last term in $J_\ve^2$. We employ \eqref{2.23.1} to infer that 
\begin{align*}
\nabla \tilde{\vv}^{T}: \boldsymbol{\xi} \otimes \boldsymbol{\xi}=(\boldsymbol{\xi} \cdot \nabla) \tilde{\vv}^{T} \cdot \boldsymbol{\xi}=0.
\end{align*}
This implies that
\begin{align*}
-\int_0^T \!\int_{\Omega}   \left(\nabla \tilde{\vv}\right)^{T} :(\bxi\otimes \bxi) |\nabla \psi_\ve|\,\mathrm{d} x \mathrm{d} t=0.
\end{align*}
Thus, using \eqref{2.16.5} and \eqref{2.16.7}, as well as \eqref{energy bound3}, we have
\begin{align*}
J_\ve^2 \le C \int_0^T E \left[\mathbf{v}_{\varepsilon}, Q_{\varepsilon} \mid \mathbf{v}, \chi\right]\, \mathrm{d} t.
\end{align*}
As a consequence of Proposition \ref{lemma exact dt relative entropy}, we have completed the estimates of the inequality.
\end{proof}

\begin{coro} \label{remark 3.1}
The term mentioned in \eqref{8.14.21} above is also positive:
\begin{align*}
\ve^2 |\partial_{t} Q_{\varepsilon}+(\mathbf{v}_{\varepsilon} \cdot \nabla) Q_{\varepsilon} |^2 -   |\HH_\ve|^2+  \Big| \HH_\ve-\ve|\nabla Q_\ve| \HH \Big|^2 \geq ~\ve^2|\partial_{t} Q_{\varepsilon}+(\mathbf{v}_{\varepsilon} \cdot \nabla) Q_{\varepsilon}+(\HH \cdot\nabla) Q_\ve|^2.
\end{align*}
\end{coro}
\begin{proof}
Based on \eqref{mean curvature app}, we know that
\begin{align*}
&~\ve^2 |\partial_{t} Q_{\varepsilon}+(\mathbf{v}_{\varepsilon} \cdot \nabla) Q_{\varepsilon} |^2 -   |\HH_\ve|^2+  \Big| \HH_\ve-\ve|\nabla Q_\ve| \HH \Big|^2\\
=&~\ve^2 | \partial_{t} Q_{\varepsilon}+(\mathbf{v}_{\varepsilon} \cdot \nabla) Q_{\varepsilon} |^2+\ve^2 |\HH|^2 |\nabla Q_\ve|^2+2\ve^2 (\HH\cdot \nabla) Q_\ve: \Big(\partial_{t} Q_{\varepsilon}+(\mathbf{v}_{\varepsilon} \cdot \nabla) Q_{\varepsilon}\Big) \\
\geq &~\ve^2 | \partial_{t} Q_{\varepsilon}+(\mathbf{v}_{\varepsilon} \cdot \nabla) Q_{\varepsilon} |^2+\ve^2 |(\HH\cdot \nabla) Q_\ve|^2+2\ve^2 (\HH\cdot \nabla) Q_\ve: \Big(\partial_{t} Q_{\varepsilon}+(\mathbf{v}_{\varepsilon} \cdot \nabla) Q_{\varepsilon}\Big)\\
=&~\ve^2|\partial_{t} Q_{\varepsilon}+(\mathbf{v}_{\varepsilon} \cdot \nabla) Q_{\varepsilon}+(\HH \cdot\nabla) Q_\ve|^2.
\end{align*}
Hence the inequality is proved.
\end{proof}

\begin{coro}  \label{remark 3.2}
Let $\mathbf{w}=\vv_\e-\vv$. Then the right-hand-side of \eqref{hope} can be written as
\begin{eqnarray} \label{8.26.1} 
\begin{split}
&C\int_0^T E \left[\mathbf{v}_{\varepsilon}, Q_{\varepsilon} \mid \mathbf{v}, \chi\right]\, \mathrm{d} t-\int_0^T \!\int_{\Omega} (\operatorname{div} \boldsymbol{\xi}) \, D d^F_\e(Q_\e) \colon (\vv_\e \cdot \nabla) Q_\e\, \mathrm{d} x \mathrm{d} t \\
&+ \int_0^T \!\int_{\Omega}   (\vv \cdot \nabla) \bxi\cdot\nabla \psi_\ve+ \nabla \tilde{\vv} : \bxi \otimes \nn_\ve |\nabla \psi_\ve|\,\mathrm{d} x \mathrm{d} t \\
=&C \int_0^T E \left[\mathbf{v}_{\varepsilon}, Q_{\varepsilon} \mid \mathbf{v}, \chi\right]\, \mathrm{d} t-\int_0^T \!\int_{\Omega} (\div \bxi) \mathbf{w} \cdot \nn_\e|\nabla \psi_\e|\, \mathrm{d} x \mathrm{d} t +\int_0^T \!\int_{\Omega}  \HH_\ve \cdot \mathbf{v} |\nabla Q_\e| \;\mathrm{d} x \mathrm{d} t \\
&-\int_0^T \!\int_{\Omega}  \HH_\ve \cdot \mathbf{v} |\nabla Q_\e| \;\mathrm{d} x \mathrm{d} t-\int_0^T \!\int_{\Omega} (\div \bxi) \vv \cdot \nn_\e|\nabla \psi_\e|\, \mathrm{d} x \mathrm{d} t \\
&+ \int_0^T \!\int_{\Omega}   \left( \vv \cdot \nabla\right) \bxi\cdot\nn_\e |\nabla \psi_\e|+ (\nn_\e\cdot\nabla) \tilde{\vv}\cdot\bxi |\nabla \psi_\e|\,\mathrm{d} x \mathrm{d} t.    
		\end{split}
\end{eqnarray}
\end{coro}
\noindent
For this identity, \eqref{ADM chain rule} and $\nabla  \tilde{\vv}:\boldsymbol{\xi} \otimes \mathbf{n}_{\varepsilon} =(\mathbf{n}_{\varepsilon} \cdot \nabla)  \tilde{\vv} \cdot \boldsymbol{\xi}$ are used again.

The terms of the last two lines on the right-hand-side of \eqref{8.26.1} are estimated below.
\begin{proposition} \label{prop 3.2}
There exists a universal constant $C>0$ which is independent of both  $T\in (0,T_0)$ and $\e $,  such that the following estimate holds  for  every $T\in (0,T_0)$:
\begin{small}
\begin{eqnarray} \label{energy3}
\begin{split}
&\int_0^T \!\int_{\Omega}  \left( \vv \cdot \nabla\right) \bxi\cdot\nn_\e |\nabla \psi_\e|+ (\nn_\e\cdot\nabla) \tilde{\vv}\cdot\bxi |\nabla \psi_\e|\,\mathrm{d} x \mathrm{d} t -\int_0^T \!\int_{\Omega}   \vv\cdot \HH_\e|\nabla Q_\e| +(\div \bxi) \vv\cdot \nn_\e|\nabla \psi_\e|\,\mathrm{d} x \mathrm{d} t \\
  \leq  & C \int_0^T E \left[\mathbf{v}_{\varepsilon}, Q_{\varepsilon} \mid \mathbf{v}, \chi\right] \, \mathrm{d} t  +\frac{\lambda}{\e}\int_0^T \!\int_{\Omega}   \left| \ve \Big(\partial_{t} Q_{\varepsilon}+(\mathbf{v}_{\varepsilon} \cdot \nabla) Q_{\varepsilon}\Big)  -(\operatorname{div} \bxi) D d^F_\ve(Q_\ve)  \right|^2\mathrm{d} x \mathrm{d} t,
\end{split}
\end{eqnarray}
\end{small}
provided that a suitably small constant $\lambda$ is chosen.
\end{proposition}
 
The following lemma will be used for the estimates presented in Proposition \ref{prop 3.2}.
\begin{lemma} \label{lemma 3.4}
For $\tilde{\vv}$ defined in \eqref{2.23.2}, the following identity holds:
\begin{align} \label{symmetry2}
&0=\int_0^T \!\int_{\Omega} (\div  \bxi)   \tilde{\vv}\cdot \nn_\e|\nabla\psi_\e|\, \mathrm{d} x \mathrm{d} t-\int_0^T \!\int_{\Omega}(\div  \tilde{\vv}) \bxi\cdot  \nn_\e|\nabla\psi_\e|\,\mathrm{d} x \mathrm{d} t- \int_0^T \!\int_{\Omega}(\tilde{\vv} \cdot \nabla) \bxi \cdot \nn_\e|\nabla\psi_\e|\, \mathrm{d} x \mathrm{d} t.
\end{align}
\end{lemma}

\begin{proof}
Using symmetry, we have $\div  \div  (\tilde{\vv}  \otimes \bxi-\bxi \otimes \tilde{\vv} )=0$ holds almost everywhere, from which we derive
\begin{align*} 
0 
&=-\int_0^T \!\int_{\Omega}  \div \Big( \div \(\tilde{\vv} \otimes \bxi-\bxi \otimes \tilde{\vv} \)\Big) \psi_\e \,\mathrm{d} x \mathrm{d} t=\int_0^T \!\int_{\Omega} \nn_\e\cdot \div \(\tilde{\vv} \otimes \bxi-\bxi \otimes \tilde{\vv} \) |\nabla\psi_\e|\,\mathrm{d} x \mathrm{d} t\nonumber\\
&=\int_0^T \!\int_{\Omega} (\div  \bxi)   \tilde{\vv}\cdot \nn_\e|\nabla\psi_\e|+(\bxi\cdot\nabla)\tilde{\vv}\cdot\nn_\e |\nabla\psi_\e|\,\mathrm{d} x \mathrm{d} t\\
&-\int_0^T \!\int_{\Omega} (\div  \tilde{\vv}) \bxi\cdot  \nn_\e|\nabla\psi_\e|+ (\tilde{\vv} \cdot \nabla) \bxi \cdot \nn_\e|\nabla\psi_\e|  \,\mathrm{d} x \mathrm{d} t.
 \end{align*}
Employing the fact \eqref{2.23.1} that $(\bxi\cdot\nabla) \tilde{\vv}=0$ in $
 \O$ to deduce $\int_0^T \!\int_{\Omega}  (\bxi\cdot\nabla)\tilde{\vv}\cdot\nn_\e |\nabla\psi_\e|\,\mathrm{d} x \mathrm{d} t=0$ and we prove \eqref{symmetry2}.
\end{proof}

\begin{proof} [Proof of Proposition \ref{prop 3.2}]
The proof relies primarily on Lemma \ref{prop2.1}, Lemma \ref{lemma:expansion 1} and Lemma \ref{lemma 3.4}. Let us deal with the second item $\int_0^T \!\int_{\Omega} (\nn_\e\cdot\nabla) \tilde{\vv}\cdot\bxi |\nabla \psi_\e|\,\mathrm{d} x \mathrm{d} t $ by using Lemma \ref{lemma:expansion 1} first.
It follows from \eqref{2.23.1} that 
\begin{align*} 
\nabla \tilde{\vv}: (\bxi-\nn_\ve) \otimes \boldsymbol{\xi}=(\boldsymbol{\xi} \cdot \nabla) \tilde{\vv} \cdot (\bxi-\nn_\ve)=0.
\end{align*}
Therefore the first term on the right-hand-side of \eqref{expansion1} can be estimated as follows:
\begin{align} \label{9.19.24}
\int_0^T \!\int_{\Omega} \nabla \tilde{\vv}: (\bxi-\nn_\ve) \otimes\nn_{\varepsilon}\left|\nabla \psi_{\varepsilon}\right|\, \mathrm{d} x \mathrm{d} t=\int_0^T \!\int_{\Omega} \nabla \tilde{\vv}: (\bxi-\nn_\ve) \otimes (\nn_{\varepsilon}-\bxi) \left|\nabla \psi_{\varepsilon}\right|\, \mathrm{d} x \mathrm{d} t.
\end{align}
The last line of \eqref{expansion1} can be computed in a manner similar to \eqref{hungry}:
\begin{small}
	\begin{align} \label{9.19.23}
&-\e\int_0^T \!\int_{\Omega} \sum_{i,j=1}^3(\nabla \tilde{\vv})_{ij}\(\p_i Q_\ve: \p_j Q_\ve   \)\, \mathrm{d} x \mathrm{d} t +\int_0^T \!\int_{\Omega} \nabla \tilde{\vv}: (\nn_\ve \otimes\nn_{\varepsilon})\left|\nabla \psi_{\varepsilon}\right|\, \mathrm{d} x \mathrm{d} t\notag\\
	=&\int_0^T \!\int_{\Omega}  \nabla \tilde{\vv}: \nn_\ve \otimes\nn_{\varepsilon}\left|\nabla \psi_{\varepsilon}\right|\,\mathrm{d} x \mathrm{d} t-\ve\int_0^T \!\int_{\Omega} \sum_{i,j=1}^3(\nabla \tilde{\vv})_{ij}(\Pi_{Q_\ve} \p_i Q_\ve :\Pi_{Q_\ve} \p_j Q_\ve)\,\mathrm{d} x \mathrm{d} t\notag\\
	&\quad-\e\int_0^T \!\int_{\Omega}  \sum_{i,j=1}^3(\nabla \tilde{\vv})_{ij} \Big((\p_i Q_\ve-\Pi_{Q_\ve} \p_i Q_\ve):(\p_j Q_\ve-\Pi_{Q_\ve} \p_j Q_\ve)\Big) \, \mathrm{d} x \mathrm{d} t\notag\\
	=&\int_0^T \!\int_{\Omega}  \nabla \tilde{\vv}: \nn_\ve \otimes\nn_{\varepsilon}\(|\nabla \psi_\ve|-\ve |\nabla Q_\ve|^2\)\,\mathrm{d} x \mathrm{d} t+\ve \int_0^T \!\int_{\Omega}  \nabla \tilde{\vv}:(\nn_\ve\otimes \nn_\ve)\(
	|\nabla Q_\ve|^2-|\Pi_{Q_\ve} \nabla Q_\ve|^2\)\, \mathrm{d} x \mathrm{d} t \notag\\
	&\quad- \e\int_0^T \!\int_{\Omega}  \sum_{i,j=1}^3(\nabla \tilde{\vv})_{ij} \Big((\p_i Q_\ve-\Pi_{Q_\ve} \p_i Q_\ve):(\p_j Q_\ve-\Pi_{Q_\ve} \p_j Q_\ve)\Big) \, \mathrm{d} x \mathrm{d} t.
	\end{align}
\end{small}
Putting \eqref{2.23.1}, \eqref{expansion1}, \eqref{9.19.24}, \eqref{9.19.23} and  Lemma \ref{prop2.1} together, we derive that
\begin{align*}
&\int_0^T \!\int_{\Omega} (\nn_\e\cdot\nabla) \tilde{\vv}\cdot\bxi |\nabla \psi_\e|\,\mathrm{d} x \mathrm{d} t=\int_0^T \!\int_{\Omega}  \nabla \tilde{\vv} : (\bxi \otimes\nn_\e)|\nabla \psi_\e| \,\mathrm{d} x \mathrm{d} t \\
	\leq &  C \int_0^T E \left[\mathbf{v}_{\varepsilon}, Q_{\varepsilon} \mid \mathbf{v}, \chi\right]\, \mathrm{d} t+\int_0^T \!\int_{\Omega}  (\div \tilde{\vv} ) \, \bxi  \cdot \nn_\e |\nabla \psi_\e| \,\mathrm{d} x \mathrm{d} t +\int_0^T \!\int_{\Omega}  \HH_\e\cdot \tilde{\vv}  |\nabla Q_\e|\,\mathrm{d} x \mathrm{d} t.
\end{align*}
And also recall \eqref{symmetry2}:
$$
0=\int_0^T \!\int_{\Omega}  (\div  \bxi)   \tilde{\vv}\cdot \nn_\e|\nabla\psi_\e|\,\mathrm{d} x \mathrm{d} t -\int_0^T \!\int_{\Omega} (\div  \tilde{\vv}) \bxi\cdot  \nn_\e|\nabla\psi_\e| \,\mathrm{d} x \mathrm{d} t - \int_0^T \!\int_{\Omega} (\tilde{\vv} \cdot \nabla) \bxi \cdot \nn_\e|\nabla\psi_\e|\,\mathrm{d} x \mathrm{d} t.
$$
Combining the above two equations, \eqref{ADM chain rule} and  \eqref{mean curvature app}, we  have
\begin{eqnarray} \label{energy2}
\begin{split}
&LHS \text{ of \eqref{energy3}} \leq  C \int_0^T E \left[\mathbf{v}_{\varepsilon}, Q_{\varepsilon} \mid \mathbf{v}, \chi\right]\, \mathrm{d} t \\
&\qquad+\int_0^T \!\int_{\Omega}  \left( \vv \cdot \nabla\right) \bxi\cdot\nn_\e |\nabla \psi_\e|  -(\div \bxi)\vv\cdot \nn_\e|\nabla \psi_\e|-\vv\cdot \HH_\e|\nabla Q_\e|\,\mathrm{d} x \mathrm{d} t \\
&\qquad+\int_0^T \!\int_{\Omega}   (\div \tilde{\vv} ) \, \bxi  \cdot \nn_\e |\nabla \psi_\e|  + \HH_\e\cdot \tilde{\vv}  |\nabla Q_\e| \,\mathrm{d} x \mathrm{d} t \\
&\qquad+\int_0^T \!\int_{\Omega}   (\div  \bxi)   \tilde{\vv}\cdot \nn_\e|\nabla\psi_\e|-(\div  \tilde{\vv}) \bxi\cdot  \nn_\e|\nabla\psi_\e|- (\tilde{\vv} \cdot \nabla) \bxi \cdot \nn_\e|\nabla\psi_\e|\,\mathrm{d} x \mathrm{d} t \\
=  &C \int_0^T E \left[\mathbf{v}_{\varepsilon}, Q_{\varepsilon} \mid \mathbf{v}, \chi\right]\, \mathrm{d} t + \int_0^T \!\int_{\Omega}   \left( (\vv-\tilde{\vv}) \cdot \nabla\right) \bxi\cdot\nn_\e |\nabla \psi_\e| \,\mathrm{d} x \mathrm{d} t  \\
&+\int_0^T \!\int_{\Omega}  (\div \bxi)(\tilde{\vv}-\vv)\cdot D d^F_\ve(Q_\ve) \colon \nabla Q_\ve-\e (\tilde{\vv}-\vv)\cdot \Big(\partial_{t} Q_{\varepsilon}+(\mathbf{v}_{\varepsilon} \cdot \nabla) Q_{\varepsilon}\Big): \nabla Q_\ve  \,\mathrm{d} x \mathrm{d} t \\
=  &C \int_0^T E \left[\mathbf{v}_{\varepsilon}, Q_{\varepsilon} \mid \mathbf{v}, \chi\right]\, \mathrm{d} t + \int_0^T \!\int_{\Omega}   \left( (\vv-\tilde{\vv}) \cdot \nabla\right) \bxi\cdot\nn_\e |\nabla \psi_\e| \,\mathrm{d} x \mathrm{d} t  \\
&+\int_0^T \!\int_{\Omega}  (\tilde{\vv}-\vv)\cdot \Big[(\div \bxi)D d^F_\ve(Q_\ve) -\e \Big(\partial_{t} Q_{\varepsilon}+(\mathbf{v}_{\varepsilon} \cdot \nabla) Q_{\varepsilon}\Big)\Big]: \nabla Q_\ve  \,\mathrm{d} x \mathrm{d} t .  
\end{split}
\end{eqnarray}
For the second term on the right-hand-side of \eqref{energy2}, we infer from  the definition of $\bxi$ and \eqref{lipschitz uu} that  
\[ \Big|(\vv-\tilde{\vv})\cdot\nabla|\bxi|^2\Big|\leq C  \min (\dd_\Gamma^2,1),\]
hence we deduce that
\begin{align*}
&\int_0^T \!\int_{\Omega}   \left( (\vv-\tilde{\vv}) \cdot \nabla\right) \bxi\cdot\nn_\e |\nabla \psi_\e| \,\mathrm{d} x \mathrm{d} t \notag\\
=&\int_0^T \!\int_{\Omega}   \left( (\vv-\tilde{\vv}) \cdot \nabla\right) |\bxi|^2 |\nabla \psi_\e| \,\mathrm{d} x \mathrm{d} t +\int_0^T \!\int_{\Omega}   \left( (\vv-\tilde{\vv}) \cdot \nabla\right) \bxi\cdot(\nn_\e-\bxi) |\nabla \psi_\e| \,\mathrm{d} x \mathrm{d} t \\
\le& C \int_0^T E \left[\mathbf{v}_{\varepsilon}, Q_{\varepsilon} \mid \mathbf{v}, \chi\right]\, \mathrm{d} t .
\end{align*}
The  remainder term  satisfies the following estimates.
\begin{align} 
&\int_0^T \!\int_{\Omega}  (\tilde{\vv}-\vv)\cdot \Big[(\div \bxi)D d^F_\ve(Q_\ve) -\e \Big(\partial_{t} Q_{\varepsilon}+(\mathbf{v}_{\varepsilon} \cdot \nabla) Q_{\varepsilon}\Big)\Big]: \nabla Q_\ve  \,\mathrm{d} x \mathrm{d} t \nonumber\\
\le& \frac{\lambda}{\e}\int_0^T \!\int_{\Omega}   \left| \ve \Big(\partial_{t} Q_{\varepsilon}+(\mathbf{v}_{\varepsilon} \cdot \nabla) Q_{\varepsilon}\Big) -(\operatorname{div} \bxi) D d^F_\ve(Q_\ve)  \right|^2\mathrm{d} x \mathrm{d} t+C \e \int_0^T \!\int_{\Omega} (\tilde{\vv}-\vv)^2 | \nabla Q_\ve|^2\,\mathrm{d} x \mathrm{d} t  \nonumber\\
\le& \frac{\lambda}{\e}\int_0^T \!\int_{\Omega}   \left| \ve \Big(\partial_{t} Q_{\varepsilon}+(\mathbf{v}_{\varepsilon} \cdot \nabla) Q_{\varepsilon}\Big) -(\operatorname{div} \bxi) D d^F_\ve(Q_\ve)  \right|^2\mathrm{d} x \mathrm{d} t+C \int_0^T E \left[\mathbf{v}_{\varepsilon}, Q_{\varepsilon} \mid \mathbf{v}, \chi\right]\, \mathrm{d} t \label{8.14.1}. 
\end{align}
Putting \eqref{energy2}-\eqref{8.14.1} together, we arrive at
\begin{small}
\begin{align*} 
&\int_0^T \!\int_{\Omega}  \left( \vv \cdot \nabla\right) \bxi\cdot\nn_\e |\nabla \psi_\e|+ (\nn_\e\cdot\nabla) \tilde{\vv}\cdot\bxi |\nabla \psi_\e|\,\mathrm{d} x \mathrm{d} t -\int_0^T \!\int_{\Omega}  \vv\cdot \HH_\e|\nabla Q_\e| +(\div \bxi) \vv\cdot \nn_\e|\nabla \psi_\e|\,\mathrm{d} x \mathrm{d} t \notag\\
  \leq  & C \int_0^T E \left[\mathbf{v}_{\varepsilon}, Q_{\varepsilon} \mid \mathbf{v}, \chi\right]\, \mathrm{d} t   +\frac{\lambda}{\e}\int_0^T \!\int_{\Omega}  \left| \ve \Big(\partial_{t} Q_{\varepsilon}+(\mathbf{v}_{\varepsilon} \cdot \nabla) Q_{\varepsilon}\Big) -(\operatorname{div} \bxi) D d^F_\ve(Q_\ve)  \right|^2\mathrm{d} x \mathrm{d} t.
\end{align*}
\end{small}
This completes the proof.
 \end{proof}

Throughout the following discussion,we will adopt the notation $$\mathbf{w}=\vv_\e-\vv$$ for convenience.
\begin{coro} \label{remark 3.3}
 As a result of Proposition \ref{lemma 3.4.26}, Corollary \ref{remark 3.1}, Corollary \ref{remark 3.2} and Proposition \ref{prop 3.2}, we claim that
\begin{align}
&\int_{\Omega} \frac{1}{2} |\mathbf{v}_\varepsilon(t)|^2+ \left[\frac{\varepsilon}{2}|\nabla Q_{\varepsilon}|^2+ \frac{1}{\varepsilon}F_\e (Q_{\varepsilon})-\boldsymbol{\xi} \cdot \nabla \psi_{\varepsilon}\right](t)\mathrm{d} x \bigg|_{t=0}^{t=T} \notag\\
&+\int_0^T \!\int_{\Omega} |\nabla \mathbf{v}_\varepsilon|^2 \mathrm{d} x \, \mathrm{d} t+\frac 1{\ve}(\frac 1{2}-\lambda)\int_0^T \!\int_{\Omega}  \left| \ve \Big(\partial_{t} Q_{\varepsilon}+(\mathbf{v}_{\varepsilon} \cdot \nabla) Q_{\varepsilon}\Big)  -(\operatorname{div} \bxi) D d^F_\ve(Q_\ve)  \right|^2\mathrm{d} x \mathrm{d} t\notag\\
		&+\frac {\ve}{2} \int_0^T \!\int_{\Omega} |\partial_{t} Q_{\varepsilon}+(\mathbf{v}_{\varepsilon} \cdot \nabla) Q_{\varepsilon}+(\HH \cdot\nabla) Q_\ve|^2\,\mathrm{d} x \mathrm{d} t \notag\\
		\le& C \int_0^T E \left[\mathbf{v}_{\varepsilon}, Q_{\varepsilon} \mid \mathbf{v}, \chi\right]\, \mathrm{d} t-\int_0^T \!\int_{\Omega} (\div \bxi) \mathbf{w} \cdot \nn_\e|\nabla \psi_\e|\, \mathrm{d} x \mathrm{d} t +\int_0^T \!\int_{\Omega}  \HH_\ve \cdot \mathbf{v} |\nabla Q_\e| \;\mathrm{d} x \mathrm{d} t \label{10.9.14}. 
\end{align}
\end{coro}

With the help of these propositions, we are able to derive the estimate in Proposition \ref{prop 3.4}. 

\begin{proof}[Proof of Proposition \ref{prop 3.4}]

Testing \eqref{9.6.4} by $\mathbf{v}$, integrating the resulting equality over $\Omega$, and using integration by parts and the divergence free condition of $\mathbf{v}$, we deduce that
\begin{eqnarray}
\begin{split} \label{8.11.5}
&\int_0^T \!\int_{\Omega} \partial_t(\mathbf{v} \cdot \mathbf{v}_\varepsilon)  -\partial_t\mathbf{v} \cdot \mathbf{v}_\varepsilon -\mathbf{v}_\varepsilon \otimes \mathbf{v}_\varepsilon :\nabla \mathbf{v} \;\mathrm{d} x \mathrm{d} t +\int_0^T \!\int_{\Omega}  \nabla \mathbf{v}_\varepsilon: \nabla \mathbf{v} \;\mathrm{d} x \mathrm{d} t \\
=&\varepsilon \int_0^T \!\int_{\Omega} \nabla Q_{\varepsilon} \odot \nabla Q_{\varepsilon}:\nabla \mathbf{v}\;\mathrm{d} x \mathrm{d} t .
\end{split}
\end{eqnarray}
Next, multiply \eqref{11.8.1} by $\mathbf{w}=\vv_\e-\vv$, integrate it over $\Omega$, and we obtain
\begin{align} \label{8.11.6}
\int_0^T \!\int_{\Omega}  \(\partial_t\mathbf{v}+(\mathbf{v}\cdot \nabla) \mathbf{v}\)\cdot \mathbf{w} \;\mathrm{d} x \mathrm{d} t +\int_0^T \!\int_{\Omega}  \nabla \mathbf{v}: \nabla \mathbf{w} \;\mathrm{d} x \mathrm{d} t =-\int_0^T \!\int_{\Gamma_t} \sigma H_{\Gamma_{t}} \mathbf{n}_{\Gamma_{t}} \cdot \mathbf{w} \,\mathrm{d} \mathcal{H}^2 \, \mathrm{d} t.
\end{align}
The regularity of the sharp interface limit solution, coupled with the condition $\operatorname{div}\mathbf{v}_{\varepsilon} = 0$, implies that
\begin{align} \label{11.12.0}
\int_0^T \!\int_{\Omega} \mathbf{v} \cdot  (\mathbf{v}_\varepsilon \cdot \nabla) \mathbf{v} \,\mathrm{d} x \, \mathrm{d} t=0.
\end{align}
Combining \eqref{8.11.5}-\eqref{11.12.0}, we get
\begin{eqnarray}
\begin{split} \label{lucky}
&\;\frac{1}{2}\, \int_{\Omega} |\mathbf{w}(T)|^2\;\mathrm{d} x-\frac{1}{2}\, \int_{\Omega} |\mathbf{w}(0)|^2\;\mathrm{d} x+\int_0^T \!\int_{\Omega} (\mathbf{w}\cdot \nabla) \mathbf{v}\cdot\mathbf{w}+|\nabla \mathbf{w}|^2 \;\mathrm{d} x \mathrm{d} t  \\
=&\frac{1}{2} \int_{\Omega} |\mathbf{v}_\varepsilon(T)|^2\mathrm{d} x-\frac{1}{2} \int_{\Omega} |\mathbf{v}_\varepsilon(0)|^2\mathrm{d} x+\int_0^T \!\int_{\Omega} |\nabla \mathbf{v}_\varepsilon|^2 \mathrm{d} x \, \mathrm{d} t \\
&-\varepsilon \int_0^T \!\int_{\Omega} \nabla Q_{\varepsilon} \odot \nabla Q_{\varepsilon}:\nabla \mathbf{v}\;\mathrm{d} x \mathrm{d} t + \int_0^T \!\int_{\Gamma_t} \sigma H_{\Gamma_{t}} \mathbf{n}_{\Gamma_{t}} \cdot \mathbf{w} \,\mathrm{d} \mathcal{H}^2\, \mathrm{d} t.
\end{split}
\end{eqnarray}

For the last term on the right hand side of \eqref{lucky}, by using \eqref{2.18.2}, we have:
\begin{small}
\begin{align}
& \int_0^T \!\int_{\Gamma_t} \sigma H_{\Gamma_{t}} \mathbf{n}_{\Gamma_{t}} \cdot \mathbf{w} \,\mathrm{d} \mathcal{H}^2\, \mathrm{d} t\overset{\eqref{2.18.2}}=-\int_0^T \!\int_{\Gamma_t}\sigma (\operatorname{div}\boldsymbol{\xi} )\mathbf{n}_{\Gamma_{t}} \cdot \mathbf{w} \,\mathrm{d} \mathcal{H}^2\, \mathrm{d} t \notag\\
=& - \int_0^T \!\int_{\Omega^-(t)} \sigma \operatorname{div}[(\operatorname{div}\boldsymbol{\xi})\mathbf{w}] \,\mathrm{d} x \mathrm{d} t = - \int_0^T \!\int_{\Omega^-(t)} \sigma (\mathbf{w}\cdot \nabla)(\operatorname{div}\boldsymbol{\xi}) \,\mathrm{d} x \mathrm{d} t = - \int_0^T \!\int_{\Omega} \sigma \chi(\mathbf{w}\cdot \nabla)(\operatorname{div}\boldsymbol{\xi}) \,\mathrm{d} x \mathrm{d} t \label{8.11.2},
\end{align}
\end{small}
where the Gauss's divergence theorem and $\operatorname{div} \mathbf{w}=0$ are used. Employing an integration by parts, the definition of $ \HH_\ve$ and $\operatorname{div} \mathbf{v}=0$ again, we have
\begin{align} \label{8.11.1}
-\varepsilon \int_0^T \!\int_{\Omega} \nabla Q_{\varepsilon} \odot \nabla Q_{\varepsilon}:\nabla \mathbf{v}\;\mathrm{d} x \mathrm{d} t&=\int_0^T \!\int_{\Omega} \varepsilon \,\Big[\frac{1}{2} \nabla (|\nabla Q_{\varepsilon}|^2)+ \Delta Q_{\varepsilon} \colon \nabla Q_{\varepsilon}\Big]\cdot \mathbf{v}\;\mathrm{d} x \mathrm{d} t\notag\\
&=-\int_0^T \!\int_{\Omega}  \HH_\ve \cdot \mathbf{v} |\nabla Q_\e| \;\mathrm{d} x \mathrm{d} t.
\end{align}
Then \eqref{lucky}-\eqref{8.11.1} give
\begin{align}
&\frac{1}{2}\, \int_{\Omega} |\mathbf{w}(T)|^2\;\mathrm{d} x+\int_0^T \!\int_{\Omega} |\nabla \mathbf{w}|^2 \mathrm{d} x \mathrm{d} t \notag\\
=&\frac{1}{2}\, \int_{\Omega} |\mathbf{w}(0)|^2\;\mathrm{d} x+\frac{1}{2} \int_{\Omega} |\mathbf{v}_\varepsilon(T)|^2\mathrm{d} x-\frac{1}{2} \int_{\Omega} |\mathbf{v}_\varepsilon(0)|^2\mathrm{d} x+\int_0^T \!\int_{\Omega} |\nabla \mathbf{v}_\varepsilon|^2 \mathrm{d} x \, \mathrm{d} t\notag\\
&+\int_0^T \!\int_{\Omega} -(\mathbf{w}\cdot\nabla) (\mathbf{v}\cdot\mathbf{w})\,\mathrm{d} x \mathrm{d} t-\int_0^T \!\int_{\Omega} \sigma \chi(\mathbf{w}\cdot \nabla)(\operatorname{div}\boldsymbol{\xi})+(\div \bxi) \mathbf{w} \cdot \nn_\e|\nabla \psi_\e|\, \mathrm{d} x \mathrm{d} t \notag\\
&+\int_0^T \!\int_{\Omega} (\div \bxi) \mathbf{w} \cdot \nn_\e|\nabla \psi_\e|\, \mathrm{d} x \mathrm{d} t-\int_0^T \!\int_{\Omega}  \HH_\ve \cdot \mathbf{v} |\nabla Q_\e| \;\mathrm{d} x \mathrm{d} t .
\end{align}
Since
\begin{align*}
\left|-\int_0^T \!\int_{\Omega} (\mathbf{w}\cdot\nabla) (\mathbf{v}\cdot\mathbf{w}) \mathrm{d} x \mathrm{d} t \right| \le \| \nabla \mathbf{v}\|_{L^\infty} \int_0^T \!\int_{\Omega} |\mathbf{w}|^2 \mathrm{d} x \mathrm{d} t  \le C \int_0^T \!\int_{\Omega} |\mathbf{w}|^2 \mathrm{d} x \mathrm{d} t
\end{align*}
and also integration by parts yields 
\begin{align*}
-\int_0^T \!\int_{\Omega} \sigma \chi(\mathbf{w}\cdot \nabla)(\operatorname{div}\boldsymbol{\xi})+(\div \bxi) \mathbf{w} \cdot \nn_\e|\nabla \psi_\e|\, \mathrm{d} x \mathrm{d} t = \int_0^T \!\int_{\Omega} (\psi_\e-\sigma \chi)(\mathbf{w}\cdot \nabla)(\operatorname{div}\boldsymbol{\xi})\, \mathrm{d} x \mathrm{d} t ,
\end{align*}
thus, by  Lemma \ref{lemma2.2} and \eqref{10.9.14} in Corollary \ref{remark 3.3}, we imply that
\begin{align*}
&\quad E \left[\mathbf{v}_{\varepsilon}, Q_{\varepsilon} \mid \mathbf{v}, \chi\right](T)+(1-\lambda)\int_0^T \!\int_{\Omega} |\nabla \mathbf{w}|^2 \mathrm{d} x \mathrm{d} t \notag\\
&+\frac 1{\ve}(\frac{1}{2}-\lambda)\int_0^T \!\int_{\Omega}  \left| \ve \Big(\partial_{t} Q_{\varepsilon}+(\mathbf{v}_{\varepsilon} \cdot \nabla) Q_{\varepsilon}\Big) -(\operatorname{div} \bxi) D d^F_\ve(Q_\ve)  \right|^2\mathrm{d} x \mathrm{d} t \notag\\
&+\frac {\ve}{2} \int_0^T \!\int_{\Omega} |\partial_{t} Q_{\varepsilon}+(\mathbf{v}_{\varepsilon} \cdot \nabla) Q_{\varepsilon}+(\HH \cdot\nabla) Q_\ve|^2\,\mathrm{d} x \mathrm{d} t \nonumber \\
\le & E \left[\mathbf{v}_{\varepsilon}, Q_{\varepsilon} \mid \mathbf{v}, \chi\right](0) +C \int_0^T E \left[\mathbf{v}_{\varepsilon}, Q_{\varepsilon} \mid \mathbf{v}, \chi\right]+E_{\mathrm{vol}}\left[Q_{\varepsilon} \mid \chi\right] \, \mathrm{d} t 
\end{align*}
and the proof is complete.
\end{proof}

\section{Estimate of the Bulk Error} \label{sec c cali}
Now we turn to derive the estimate of bulk error as described in \eqref{area diff}. 
\begin{proposition}\label{lem a est}
Let $E_{\mathrm{vol}}\left[Q_{\varepsilon} \mid \chi\right]$ be defined as in \eqref{area diff}. Then there exist a generic constant $C>0$ and a small enough $\lambda>0$, such that for any $T \in [0,T_0]$, the following estimate holds. 
\begin{align}\label{energy8}
E_{\mathrm{vol}}\left[Q_{\varepsilon} \mid \chi\right](T)  \leq& E_{\mathrm{vol}}\left[Q_{\varepsilon} \mid \chi\right](0) +C \int_0^T E \left[\mathbf{v}_{\varepsilon}, Q_{\varepsilon} \mid \mathbf{v}, \chi\right]+E_{\mathrm{vol}}\left[Q_{\varepsilon} \mid \chi\right] \, \mathrm{d} t\nonumber\\
&+\frac{\e}{8}\int_0^T \!\int_{\Omega}  \(\p_t Q_\e+(\HH+\vv_\e) \cdot \nabla Q_\e\)^{2} \,\mathrm{d} x \mathrm{d} t +\lambda \int_0^T \!\int_{\Omega}  |\nabla \mathbf{v}_{\varepsilon}-\nabla \mathbf{v}|^2 \,\mathrm{d} x \mathrm{d} t.
\end{align}
\end{proposition}

\begin{proof}
Firstly, we define $\Lambda(x,t)\triangleq \mp 1$ in $\Omega^\pm$. Inspired by \cite{liu 2}, we use the decomposition
\begin{align} \label{8.9.3}
2(\psi_{\varepsilon}-\sigma \chi)=2 \psi_{\varepsilon}-\sigma-\sigma \Lambda = 2 (\psi_{\varepsilon}-\sigma )^{+} + (\sigma -2 (\psi_{\varepsilon}-\sigma )^{-}-\sigma \Lambda ) ,
\end{align}
and divide $E_{\mathrm{vol}}\left[Q_{\varepsilon} \mid \chi\right] $ into two non-negative parts:
\begin{align}
g_{\varepsilon}(t) & \triangleq \int_{\Omega} (\psi_{\varepsilon}-\sigma )^{+} |\vartheta(d_\Gamma)| \,\mathrm{d} x, \label{8.9.1}
\end{align}
and
\begin{align}
h_{\varepsilon}(t) & \triangleq \int_{\Omega} \left(\sigma \Lambda-[\sigma-2 (\psi_{\varepsilon}-\sigma )^{-}]\right) \vartheta (d_\Gamma) \,\mathrm{d} x \label{8.9.2}.
\end{align}
The non-negativity of \eqref{8.9.1} is obvious. As for \eqref{8.9.2}, since $(\psi_{\varepsilon}-\sigma)^{-} \in[0,\sigma]$, it follows that  the range of $[\sigma-2(\psi_{\varepsilon}-\sigma)^{-}]$ is $[-\sigma,\sigma]$. By employing  $\vartheta \Lambda=|\vartheta|$, we infer the non-negativity of  \eqref{8.9.2}  and deduce that
$$
h_{\varepsilon}(t) = \int_{\Omega} \left|\sigma -2(\psi_{\varepsilon}-\sigma )^{-}-\sigma \Lambda\right| |\vartheta(d_\Gamma)| \,\mathrm{d} x.
$$

Next we are devoted to derive the evolution of $g_\e$ and $h_\e$.
The evolution of \eqref{8.9.1} can be obtained from \eqref{8.11.7} that
\begin{align*}
g_{\varepsilon}(T) {=} 
& g_{\varepsilon}(0) +\int_0^T \! \int_{\{\psi_{\varepsilon}>\sigma \}}D d^F_\e(Q_\e) \colon \p_t Q_\e(x,t) |\vartheta(d_\Gamma)| \,\mathrm{d} x \mathrm{d} t \\
&-\int_0^T \!\int_{\Omega} (\psi_{\varepsilon}-\sigma )^{+} (\HH+\tilde{\vv}) \cdot \nabla |\vartheta(d_\Gamma)| \,\mathrm{d} x \mathrm{d} t\\
&+\int_0^T \!\int_{\Omega} (\psi_{\varepsilon}-\sigma )^{+} \Big(\partial_t |\vartheta(d_\Gamma)|+(\HH+\tilde{\vv}) \cdot \nabla |\vartheta(d_\Gamma)|\Big) \,\mathrm{d} x \mathrm{d} t \\
=&g_{\varepsilon}(0) +\int_0^T \!\int_{\{\psi_{\varepsilon}>\sigma \}}D d^F_\e(Q_\e) \colon \p_t Q_\e(x,t) |\vartheta(d_\Gamma)| \,\mathrm{d} x \mathrm{d} t\\
&+\int_0^T \!\int_{\Omega} (\HH+\tilde{\vv}) \cdot \nabla (\psi_{\varepsilon}-\sigma )^{+} |\vartheta(d_\Gamma)| \,\mathrm{d} x \mathrm{d} t\\
& +\int_0^T \!\int_{\Omega} (\psi_{\varepsilon}-\sigma )^{+} \div (\HH+\tilde{\vv})  \,|\vartheta(d_\Gamma)| \,\mathrm{d} x \mathrm{d} t\\
& +\int_0^T \!\int_{\Omega} (\psi_{\varepsilon}-\sigma )^{+} \Big(\partial_t |\vartheta(d_\Gamma)|+(\HH+\tilde{\vv}) \cdot \nabla |\vartheta(d_\Gamma)|\Big) \,\mathrm{d} x \mathrm{d} t\\
\triangleq& g_{\varepsilon}(0) +I_1+I_2+I_3+I_4.
\end{align*}
We can  estimate  $I_1$ by
\begin{align*} 
I_1=&\int_0^T \!\int_{\{\psi_{\varepsilon}>\sigma \}} \frac{D d^F_\e(Q_\e)}{|D d^F_\e(Q_\e)|} \colon \Big(\p_t Q_\e+(\HH+\vv_\e) \cdot \nabla Q_\e\Big) |D d^F_\e(Q_\e)||\vartheta(d_\Gamma)| \,\mathrm{d} x \mathrm{d} t\\
&-\int_0^T \!\int_{\Omega}  D d^F_\e(Q_\e) \colon \Big( (\HH+\vv_\e) \cdot \nabla Q_\e \Big) |\vartheta(d_\Gamma)| \chi_{\{\psi_{\varepsilon}>\sigma \}}  \,\mathrm{d} x \mathrm{d} t\\
\stackrel{\eqref{ADM chain rule}} {=}&\int_0^T \!\int_{\{\psi_{\varepsilon}>\sigma \}} \frac{D d^F_\e(Q_\e)}{|D d^F_\e(Q_\e)|} \colon \Big(\p_t Q_\e+(\HH+\vv_\e) \cdot \nabla Q_\e\Big) |D d^F_\e(Q_\e)||\vartheta(d_\Gamma)| \,\mathrm{d} x \mathrm{d} t\\
&-\int_0^T \!\int_{\Omega} (\HH+\vv_\e) \cdot \nabla (\psi_{\varepsilon}-\sigma )^{+} |\vartheta(d_\Gamma)| \,\mathrm{d} x \mathrm{d} t.
\end{align*}
Then we rearrange $I_1$ and $I_2$ as follows:
\begin{align*} 
I_1+I_2 =& \int_0^T \!\int_{\{\psi_{\varepsilon}>\sigma \}} \frac{D d^F_\e(Q_\e)}{|D d^F_\e(Q_\e)|} \colon \Big(\p_t Q_\e+(\HH+\vv_\e) \cdot \nabla Q_\e\Big) |D d^F_\e(Q_\e)||\vartheta(d_\Gamma)| \,\mathrm{d} x \mathrm{d} t \\
&+\int_0^T \!\int_{\Omega} (\tilde{\vv}-\vv) \cdot \nabla (\psi_{\varepsilon}-\sigma )^{+} |\vartheta(d_\Gamma)| \,\mathrm{d} x \mathrm{d} t +\int_0^T \!\int_{\Omega} (\vv-\vv_\e) \cdot \nabla (\psi_{\varepsilon}-\sigma )^{+} |\vartheta(d_\Gamma)| \,\mathrm{d} x \mathrm{d} t\\
\triangleq &J_1+J_2+J_3.
\end{align*}
It follows from \eqref{sharp lip d} and \eqref{energy bound3} that 
\begin{align} 
J_1 \leq& \frac{\e}{16}\int_0^T \!\int_{\Omega}  \Big(\p_t Q_\e+(\HH+\vv_\e) \cdot \nabla Q_\e\Big)^{2} \,\mathrm{d} x \mathrm{d} t +C \int_0^T \!\int_{\Omega} \e^{-1} { F_\ve(Q_\ve)} \min (\dd_\Gamma^2,1) \,\mathrm{d} x \mathrm{d} t \nonumber \\
\leq& C \int_0^T E \left[\mathbf{v}_{\varepsilon}, Q_{\varepsilon} \mid \mathbf{v}, \chi\right](t) \, \mathrm{d} t+\frac{\e}{16} \int_0^T \!\int_{\Omega}  \Big(\p_t Q_\e+(\HH+\vv_\e) \cdot \nabla Q_\e\Big)^{2} \,\mathrm{d} x \mathrm{d} t \label{8.7.24}.
\end{align}
To estimate $J_2$, \eqref{lipschitz uu} and \eqref{sharp lip d} imply
\begin{align} 
J_2=& \int_0^T \!\int_{\Omega} (\tilde{\vv}-\vv) \cdot \nabla (\psi_{\varepsilon}-\sigma )^{+} |\vartheta(d_\Gamma)| \,\mathrm{d} x \mathrm{d} t\notag\\
=&\int_0^T \!\int_{\Omega} D d^F_\e(Q_\e) \colon \Big( (\tilde{\vv}-\vv) \cdot \nabla Q_\e \Big) |\vartheta(d_\Gamma)| \chi_{\{\psi_{\varepsilon}>\sigma \}} \,\mathrm{d} x \mathrm{d} t \nonumber \\
\leq& \int_0^T \!\int_{\Omega} \(\frac{\varepsilon}2 \left|\nabla Q_\ve\right| ^{2}+\frac1 {\varepsilon}{ F_\ve(Q_\ve)}\) \min\(\dd_\Gamma^2,1\) \,\mathrm{d} x \mathrm{d} t  \leq C \int_0^T E \left[\mathbf{v}_{\varepsilon}, Q_{\varepsilon} \mid \mathbf{v}, \chi\right](t) \, \mathrm{d} t.
\end{align}
Moreover,  from Lemma \ref{lemma2.2}, we have
\begin{align} 
J_3=& \int_0^T \!\int_{\Omega} (\vv-\vv_\e) \cdot \nabla (\psi_{\varepsilon}-\sigma )^{+} |\vartheta(d_\Gamma)| \,\mathrm{d} x \mathrm{d} t=\int_0^T \!\int_{\Omega} (\vv_\e-\vv) (\psi_{\varepsilon}-\sigma )^{+} \nabla |\vartheta(d_\Gamma)| \,\mathrm{d} x \mathrm{d} t \nonumber \\
\leq& C \int_0^T \!\int_{\Omega} \Big((\psi_{\varepsilon}-\sigma )^{+} |\vartheta(d_\Gamma)|+ |\mathbf{v}_{\varepsilon}-\mathbf{v}|^2 \Big)\,\mathrm{d} x \mathrm{d} t+\lambda \int_0^T \!\int_{\Omega}  |\nabla \mathbf{v}_{\varepsilon}-\nabla \mathbf{v}|^2 \,\mathrm{d} x \mathrm{d} t \nonumber \\
\leq& C \int_0^T g_\e (t)\, \mathrm{d} t+C \int_0^T \!\int_{\Omega} |\mathbf{v}_{\varepsilon}-\mathbf{v}|^2 \,\mathrm{d} x \mathrm{d} t+\lambda \int_0^T \!\int_{\Omega}  |\nabla \mathbf{v}_{\varepsilon}-\nabla \mathbf{v}|^2 \,\mathrm{d} x \mathrm{d} t.
\end{align}
Using \eqref{7.25.1}, \eqref{regular tilde v}, \eqref{9.17.5} and \eqref{9.17.6}, it is evident that 
\begin{align}
I_3=&\int_0^T \!\int_{\Omega} (\psi_{\varepsilon}-\sigma )^{+} \div (\HH+\tilde{\vv})  \,|\vartheta(d_\Gamma)| \,\mathrm{d} x \mathrm{d} t \leq C \int_0^T g_\e (t)\, \mathrm{d} t,\\
I_4=&\int_0^T \!\int_{\Omega} (\psi_{\varepsilon}-\sigma )^{+} \left(\partial_t |\vartheta(d_\Gamma)|+(\HH+\tilde{\vv}) \cdot \nabla |\vartheta(d_\Gamma)|\right) \,\mathrm{d} x \mathrm{d} t\leq C \int_0^T g_\e (t)\, \mathrm{d} t \label{7.25.4}.
\end{align}
Consequently, in view of \eqref{8.7.24}- \eqref{7.25.4}, we arrive at 
\begin{align} \label{8.9.4}
g_{\varepsilon}(T) \le& g_{\varepsilon}(0)+C \int_0^T E \left[\mathbf{v}_{\varepsilon}, Q_{\varepsilon} \mid \mathbf{v}, \chi\right](t)\, \mathrm{d} t+C \int_0^T g_\e (t)\, \mathrm{d} t \nonumber\\
&+\frac{\e}{16}\int_0^T \!\int_{\Omega}  \Big(\p_t Q_\e+(\HH+\vv_\e) \cdot \nabla Q_\e\Big)^{2} \,\mathrm{d} x \mathrm{d} t+\lambda \int_0^T \!\int_{\Omega}  |\nabla \mathbf{v}_{\varepsilon}-\nabla \mathbf{v}|^2 \,\mathrm{d} x \mathrm{d} t.
\end{align}

Similarly to $g_\e$, the evolution of \eqref{8.9.2} can be calculated as follows.
\begin{align*}
h_{\varepsilon}(T)=&h_{\varepsilon}(0)+ \int_0^T \!\int_{\{\psi_{\varepsilon}<\sigma \}} -2 \partial_t \psi_\e \vartheta (d_\Gamma ) \,\mathrm{d} x \mathrm{d} t\\
&-\int_0^T \!\int_{\Omega} (\HH+\tilde{\vv}) \cdot \nabla \vartheta (d_\Gamma ) \left(\sigma \Lambda-[\sigma-2 (\psi_{\varepsilon}-\sigma )^{-}]\right)\,\mathrm{d} x \mathrm{d} t\\
& +\int_0^T \!\int_{\Omega} \vartheta^{\prime} (d_\Gamma ) \left(\partial_t d_\Gamma +(\HH+\tilde{\vv}) \cdot \nabla d_\Gamma \right) \left(\sigma \Lambda-[\sigma-2 (\psi_{\varepsilon}-\sigma )^{-}]\right)\,\mathrm{d} x \mathrm{d} t \\
=&h_{\varepsilon}(0)+\int_0^T \!\int_{\{\psi_{\varepsilon}<\sigma \}} -2 D d^F_\e(Q_\e) \colon \p_t Q_\e(x,t)  \vartheta (d_\Gamma )\,\mathrm{d} x \mathrm{d} t\\
&+\int_0^T \!\int_{\Omega} (\HH+\tilde{\vv}) \cdot \nabla  \left(\sigma \Lambda-[\sigma-2 (\psi_{\varepsilon}-\sigma )^{-}]\right) \vartheta (d_\Gamma ) \,\mathrm{d} x \mathrm{d} t\\
& +\int_0^T \!\int_{\Omega} \left(\sigma \Lambda-[\sigma-2 (\psi_{\varepsilon}-\sigma )^{-}]\right) \div (\HH+\tilde{\vv})  \vartheta (d_\Gamma ) \,\mathrm{d} x \mathrm{d} t \\
& +\int_0^T \!\int_{\Omega} \vartheta^{\prime} (d_\Gamma ) \left(\partial_t d_\Gamma +(\HH+\tilde{\vv}) \cdot \nabla d_\Gamma \right) \left(\sigma \Lambda-[\sigma-2 (\psi_{\varepsilon}-\sigma )^{-}]\right)\,\mathrm{d} x \mathrm{d} t \\
\triangleq&h_{\varepsilon}(0)+K_1+K_2+K_3+K_4.
\end{align*}
For  $K_1$ and $K_2$,
\begin{align*} 
K_1+K_2= & \int_0^T \!\int_{\{\psi_{\varepsilon}<\sigma \}} -2 \frac{D d^F_\e(Q_\e)}{|D d^F_\e(Q_\e)|} \colon \Big(\partial_t Q_\e+(\HH+\vv_\e) \cdot \nabla Q_\e \Big) |D d^F_\e(Q_\e)| \vartheta (d_\Gamma )\,\mathrm{d} x \mathrm{d} t \\
&-2\int_0^T \!\int_{\{\psi_{\varepsilon}<\sigma \}} (\tilde{\vv}-\vv) \cdot \left(D d^F_\e(Q_\e) \colon \nabla Q_{\varepsilon}\right) \vartheta (d_\Gamma ) \,\mathrm{d} x \mathrm{d} t \\
&-2\int_0^T \!\int_{\{\psi_{\varepsilon}<\sigma \}} (\vv-\vv_\e) \cdot \left(D d^F_\e(Q_\e) \colon \nabla Q_{\varepsilon}\right) \vartheta (d_\Gamma ) \,\mathrm{d} x \mathrm{d} t.
\end{align*}
Repeat the discussions above and we obtain
\begin{align} \label{8.9.5}
h_{\varepsilon}(T) \le& h_{\varepsilon}(0) + C \int_0^T E \left[\mathbf{v}_{\varepsilon}, Q_{\varepsilon} \mid \mathbf{v}, \chi\right](t)+ h_\e (t)\, \mathrm{d} t+\frac{\e}{16}\int_0^T \!\int_{\Omega}  (\p_t Q_\e+(\HH+\vv_\e) \cdot \nabla Q_\e)^{2} \,\mathrm{d} x \mathrm{d} t \nonumber\\
&+\lambda \int_0^T \!\int_{\Omega}  |\nabla \mathbf{v}_{\varepsilon}-\nabla \mathbf{v}|^2 \,\mathrm{d} x \mathrm{d} t.
\end{align}

Using \eqref{8.9.3}, \eqref{8.9.4} and \eqref{8.9.5}, we conclude
\begin{align*}
E_{\mathrm{vol}}\left[Q_{\varepsilon} \mid \chi\right](T)  \leq& E_{\mathrm{vol}}\left[Q_{\varepsilon} \mid \chi\right](0) +C \int_0^T E \left[\mathbf{v}_{\varepsilon}, Q_{\varepsilon} \mid \mathbf{v}, \chi\right](t)+E_{\mathrm{vol}}\left[Q_{\varepsilon} \mid \chi\right] \, \mathrm{d} t\\
&+\frac{\e}{8}\int_0^T \!\int_{\Omega}  \Big(\p_t Q_\e+(\HH+\vv_\e) \cdot \nabla Q_\e\Big)^{2} \,\mathrm{d} x \mathrm{d} t +\lambda \int_0^T \!\int_{\Omega}  |\nabla \mathbf{v}_{\varepsilon}-\nabla \mathbf{v}|^2 \,\mathrm{d} x \mathrm{d} t
\end{align*}
and finish the proof.
 \end{proof}
 

\section{Proofs of Main Theorems}

Utilizing the prior estimates derived in the preceding sections, we are now prepared to prove Theorem \ref{thm1.1}.

\begin{proof} [Proof of Theorem \ref{thm1.1}]
Applying Proposition \ref{prop 3.4} and Proposition \ref{lem a est}, we have
\begin{align*}
 & \quad E \left[\mathbf{v}_{\varepsilon}, Q_{\varepsilon} \mid \mathbf{v}, \chi\right](T)+ E_{\mathrm{vol}}\left[Q_{\varepsilon} \mid \chi\right](T)+\frac{1}{2} \int_0^T \!\int_{\Omega} |\nabla \mathbf{v}_{\varepsilon}-\nabla \mathbf{v}|^2 \mathrm{d} x \mathrm{d} t \notag\\
&+\frac 1{4\ve}\int_0^T \!\int_{\Omega}  \left| \ve \Big(\partial_{t} Q_{\varepsilon}+(\mathbf{v}_{\varepsilon} \cdot \nabla) Q_{\varepsilon}\Big) -(\operatorname{div} \bxi) D d^F_\ve(Q_\ve)  \right|^2\mathrm{d} x \mathrm{d} t \\
&+\frac {\ve}{4} \int_0^T \!\int_{\Omega} |\partial_{t} Q_{\varepsilon}+(\mathbf{v}_{\varepsilon} \cdot \nabla) Q_{\varepsilon}+(\HH \cdot\nabla) Q_\ve|^2\,\mathrm{d} x \mathrm{d} t\nonumber \\
\le & E \left[\mathbf{v}_{\varepsilon}, Q_{\varepsilon} \mid \mathbf{v}, \chi\right](0)+ E_{\mathrm{vol}}\left[Q_{\varepsilon} \mid \chi\right](0)+C\int_0^T E \left[\mathbf{v}_{\varepsilon}, Q_{\varepsilon} \mid \mathbf{v}, \chi\right]+E_{\mathrm{vol}}\left[Q_{\varepsilon} \mid \chi\right] \, \mathrm{d} t,
\end{align*}
after choosing $\lambda$ to be suitably small.
Therefore, by Gronwall's inequality, we conclude that
\begin{align*}
 &\quad E \left[\mathbf{v}_{\varepsilon}, Q_{\varepsilon} \mid \mathbf{v}, \chi\right](t)+E_{\mathrm{vol}}\left[Q_{\varepsilon} \mid \chi\right](t)+\frac{1}{2} \int_0^T \!\int_{\Omega} |\nabla \mathbf{v}_{\varepsilon}-\nabla \mathbf{v}|^2 \mathrm{d} x \mathrm{d} t \notag\\
&+\frac 1{4\ve} \int_0^T \!\int_{\Omega} \left| \ve \Big(\partial_{t} Q_{\varepsilon}+(\mathbf{v}_{\varepsilon} \cdot \nabla) Q_{\varepsilon}\Big) -(\operatorname{div} \bxi) D d^F_\ve(Q_\ve)  \right|^2\mathrm{d} x \mathrm{d} t \notag\\
&+\frac {\ve}{4}  \int_0^T \!\int_{\Omega} |\partial_{t} Q_{\varepsilon}+(\mathbf{v}_{\varepsilon} \cdot \nabla) Q_{\varepsilon}+(\HH \cdot\nabla) Q_\ve|^2\,\mathrm{d} x \mathrm{d} t \nonumber \\
\le & (1+Ct e^{Ct})\big(E \left[\mathbf{v}_{\varepsilon}, Q_{\varepsilon} \mid \mathbf{v}, \chi\right](0)+E_{\mathrm{vol}}\left[Q_{\varepsilon} \mid \chi\right](0)\big) \le C \e
\end{align*}
as \eqref{initial} is satisfied. The proof of Theorem \ref{thm1.1} is done.
\end{proof}


The following property is crucial to proving Theorem \ref{main thm}.
\begin{proposition}(see \cite{liu})
As the uniform estimate \eqref{basic constant1} holds and $\Omega^+(t)$ is a smooth simply-connected domain, there exists a subsequence of $\ve_k>0$  such that
\begin{subequations}\label{global control}
\begin{align}\label{global control1}
\left[\p_t Q_{\ve_k},Q_{\ve_k} \right]=\left[\p_t Q_{\ve_k}-\Pi_{Q_{\ve_k}}\p_t Q_{\ve_k},Q_{\ve_k} \right]\xrightarrow{k\to \infty}& \bar{S}_0(x,t)~\text{weakly in}~  L^2(0,T;L^2(\Omega)),\\
\left[\p_j Q_{\ve_k},Q_{\ve_k} \right]=\left[\p_j Q_{\ve_k}-\Pi_{Q_{\ve_k}}\p_j Q_{\ve_k},Q_{\ve_k} \right]\xrightarrow{k\to \infty}& \bar{S}_j(x,t)~\text{weakly-star in}~  L^\infty(0,T;L^2(\Omega))
\label{global control t}
\end{align}
\end{subequations}
holds for  $1\leq j\leq 3$. Furthermore, 
\begin{subequations}\label{weak strong convergence}
\begin{align}
\p_t Q_{\ve_k}\xrightarrow{ k\to\infty } \p_t Q&,~\text{weakly in}~  L^2(0,T;L^2_{loc}(\Omega^\pm(t))),\label{deri con2}\\
\nabla Q_{\ve_k}\xrightarrow{k\to\infty }  \nabla Q&,~\text{weakly in}~  L^\infty(0,T;L^2_{loc}(\Omega^\pm(t))),\label{deri con}\\
  Q_{\ve_k}\xrightarrow{k\to\infty }    Q & ,~\text{strongly in}~  C([0,T];L^2_{loc}(\Omega^\pm(t))),\label{deri con1}
\end{align}
\end{subequations}
where  $Q=Q(x,t)$ is defined by
\begin{align}\label{limuni}
&Q (x,t)=s^\pm \(\mathbf{u}(x,t) \otimes \mathbf{u}(x,t)-\frac 13I_3\)~\text{a.e.}~(x,t)\in \Omega^\pm_T\end{align}
and $\mathbf{u}$ is a unit vector field satisfies the regularity estimates
\begin{equation}\label{orientation integrable}
\mathbf{u}\in L^\infty(0,T;H^1(\Omega^+(t);\BS))\cap H^1(0,T;L^2(\Omega^+(t);\BS))\cap C([0,T];L^2(\Omega^+(t);\BS)).
\end{equation}
\end{proposition}
\begin{proof}
By utilizing \eqref{basic constant1}, the corresponding proof can be found in Proposition 5.2 in \cite{liu}. The necessity of the condition $\Omega^+(t)$ being a smooth, simply-connected domain is discussed in detail in \cite[Section 3.2]{maiersaupe}.
\end{proof}

Based on this proposition, we are ready to prove the Theorem \ref{main thm}.

  \begin{proof}[Proof of Theorem \ref{main thm}]
Throughout the following process, we using the notation  $A:B=\tr A^T B$ where $A$ and $B$ are ${3\times 3}$ matrices.
 We associate each testing vector field  $\boldsymbol{\varphi}(x,t)=(\varphi_1,\varphi_2,\varphi_3)\in C^1(\overline{\O_T},\R^3)$ with a matrix-valued function
\begin{equation}\label{wedge matrix}
\Phi(x,t)=\begin{pmatrix}
0&\varphi_3 &-\varphi_2\\
-\varphi_3 & 0 & \varphi_1\\
\varphi_2 & -\varphi_1 & 0
\end{pmatrix}.
\end{equation}
It is important to emphasize again that $[D  F(Q_{\ve_k}), Q_{\ve_k}]=0$, which inspires us to  apply the anti-symmetric product   $[\cdot, Q_{\ve_k}] $ to    \eqref{10.19.1} and integration by parts to derive
\begin{small}
\begin{align}
0=&\int_0^T \!\int_{\Omega} \left[\p_t Q_{\ve_k},  Q_{\ve_k}\right] :\Phi  \, \mathrm{d} x \mathrm{d} t + \int_0^T \!\int_{\Omega}  \sum_{j=1}^3 [\p_j Q_{\ve_k},  Q_{\ve_k}]:\p_j \Phi \, \mathrm{d} x \mathrm{d} t + \int_0^T \!\int_{\Omega}  \sum_{j=1}^3 {v_{\e_k}}_j[\p_j Q_{\ve_k},  Q_{\ve_k}]:\Phi \, \mathrm{d} x \mathrm{d} t \nonumber\\
=& \sum_{\pm}\int_0^T\int_{\Omega^\pm(t)\backslash \Gamma_t(\delta)}\(\left[\p_t Q_{\ve_k},  Q_{\ve_k}\right] :\Phi  +  \sum_{j=1}^3[\p_j Q_{\ve_k},  Q_{\ve_k}]:\p_j \Phi +\sum_{j=1}^3 {v_{\e_k}}_j[\p_j Q_{\ve_k},  Q_{\ve_k}]:\Phi \)\, \mathrm{d} x \mathrm{d} t\nonumber\\
 &+\int_0^T\int_{  \Gamma_t(\delta)}\(\left[\p_t Q_{\ve_k},  Q_{\ve_k}\right] :\Phi  +  \sum_{j=1}^3[\p_j Q_{\ve_k},  Q_{\ve_k}]:\p_j \Phi +\sum_{j=1}^3 {v_{\e_k}}_j[\p_j Q_{\ve_k},  Q_{\ve_k}]:\Phi  \)\, \mathrm{d} x \mathrm{d} t.
\end{align}
\end{small}
When $k\to \infty$, we obtain from \eqref{global control} - \eqref{limuni} that
\begin{align}
&\int_0^T\int_{\Omega^+(t)\backslash \Gamma_t(\delta)}\(\left[\p_t Q,  Q\right] :\Phi  +  \sum_{j=1}^3[\p_j Q,  Q]:\p_j \Phi+\sum_{j=1}^3 v_j[\p_j Q,  Q]:\Phi \)\, \mathrm{d} x \mathrm{d} t\nonumber\\
&\qquad +\int_0^T\int_{  \Gamma_t(\delta)}\(\bar{S}_0 :\Phi  + \sum_{j=1}^3 \bar{S}_j:\p_j \Phi +\sum_{j=1}^3 v_j\bar{S}_j:\Phi \)\, \mathrm{d} x \mathrm{d} t=0.
\end{align}
Notice that the identity 
\[\p_i \mathbf{u}\otimes \mathbf{u}-\mathbf{u}\otimes \p_i \mathbf{u}=\begin{pmatrix}
0& (\p_i \mathbf{u}\wedge \mathbf{u})_3 &-(\p_i \mathbf{u}\wedge \mathbf{u})_2\\
-(\p_i \mathbf{u}\wedge \mathbf{u})_3 & 0 & (\p_i \mathbf{u}\wedge \mathbf{u})_1\\
(\p_i \mathbf{u}\wedge \mathbf{u})_2 & -(\p_i \mathbf{u}\wedge \mathbf{u})_1 & 0
\end{pmatrix}\label{wedge matrix new}\]
holds for $i=0,1,2,3$,  where $(\p_i \mathbf{u}\wedge \mathbf{u})_k$ is the $k$-th component of $\p_i \mathbf{u}\wedge \mathbf{u}$ and $\p_0 \triangleq \p_t$. Since $\uu$ is an unite vector and employing \eqref{limuni} and \eqref{wedge matrix}, we further verify that
\begin{align*}
[\p_t Q,  Q]:\Phi&=s_+^2\(\p_t \mathbf{u}\otimes \mathbf{u}-\mathbf{u}\otimes \p_t \mathbf{u}\): \Phi=2s_+^2\p_t \mathbf{u}\wedge \mathbf{u}\cdot \boldsymbol{\varphi},\\
[\p_j Q,   Q]:\p_j\Phi&=s_+^2\(\p_j \mathbf{u}\otimes \mathbf{u}-\mathbf{u}\otimes \p_j \mathbf{u}\): \p_j\Phi=2s_+^2\p_j \mathbf{u}\wedge \mathbf{u}\cdot \p_j \boldsymbol{\varphi}.
\end{align*} 
Thereby, we achieve the estimates
\begin{align}
&2s_+^2 \int_0^T\int_{\Omega^+(t)\backslash \Gamma_t(\delta)}\(\p_t \mathbf{u}\wedge \mathbf{u}\cdot \boldsymbol{\varphi} +\sum_{j=1}^3(\p_j  \mathbf{u}\wedge \mathbf{u})\cdot \p_j \boldsymbol{\varphi}+\sum_{j=1}^3 v_j(\p_j  \mathbf{u}\wedge \mathbf{u})\cdot \boldsymbol{\varphi}\)\, \mathrm{d} x \mathrm{d} t\nonumber\\
&\qquad   + \int_0^T\int_{  \Gamma_t(\delta)}\(\bar{S}_0 :\Phi  +  \sum_{j=1}^3\bar{S}_j:\p_j \Phi +\sum_{j=1}^3 v_j\bar{S}_j:\Phi\)\, \mathrm{d} x \mathrm{d} t=0.
\end{align}

By virtue of \eqref{orientation integrable}, it implies the absolute continuity of $\p_t \mathbf{u}\wedge \mathbf{u}$ and $\nabla \mathbf{u}\wedge \mathbf{u}$ in $\O^+_T$. By using \eqref{global control}, one has the absolute continuity of $\{\bar{S}_i\}_{0\leq i\leq 3}$ in $\O_T$. Taking the limit  $\delta\to 0$ in the above identity gives
 \begin{equation}
 \int_0^T\int_{\Omega^+(t)}\p_t \mathbf{u}\wedge \mathbf{u}\cdot \boldsymbol{\varphi}\, \mathrm{d} x \mathrm{d} t +\int_0^T\int_{\Omega^+(t)}\sum_{j=1}^3(\p_j  \mathbf{u}\wedge \mathbf{u})\cdot \p_j \boldsymbol{\varphi}\, \mathrm{d} x \mathrm{d} t  +\int_0^T\int_{\Omega^+(t)}\sum_{j=1}^3 v_j(\p_j  \mathbf{u}\wedge \mathbf{u})\cdot \boldsymbol{\varphi}\, \mathrm{d} x \mathrm{d} t   =0.\end{equation}
 The proof of Theorem \ref{main thm} is completed.
 \end{proof}

\noindent{\bf Acknowledgments:} The author thanks  Professor Yuning Liu for introducing her to his work with Tim Laux \cite{liu} and guiding her toward the relevant reference \cite{Sebastian Hensel}. 
 Additionally, thanks to Professor Feng Xie for the valuable comments and suggestions provided for this article.

\bibliography{/Users/yl67/Dropbox/cfs.bib}
\bibliographystyle{abbrv}

\end{document}